\documentclass{amsproc}

\usepackage{hyperref}
\usepackage{array}

\usepackage[utf8]{inputenc} 
\usepackage{ae,aecompl}

\usepackage{amsmath} 
\usepackage{amsthm} 
\usepackage{amssymb} 

\usepackage{mathrsfs} 

\usepackage{amscd} 
\usepackage[all,cmtip]{xy}

\usepackage{amstext} 
\usepackage{amsfonts}

\usepackage{enumitem}

\usepackage{bm}

\usepackage{mathdots} 

\usepackage{color}

\DeclareMathAlphabet{\mathpzc}{OT1}{pzc}{m}{it} 

\definecolor{zielony}{rgb}{0.5, 0.9, 0.1}
\definecolor{czerwony}{rgb}{0.9, 0.2, 0.1}
\definecolor{niebieski}{rgb}{0.3, 0.1, 0.9}

\date{\today}

\numberwithin{equation}{section}

\newtheorem{conjecture}[equation]{Conjecture}
\newtheorem{theorem}[equation]{Theorem}
\newtheorem*{theorem*}{Theorem}

\newtheorem{proposition}[equation]{Proposition}
\newtheorem{lemma}[equation]{Lemma}

\newtheorem{corollary}[equation]{Corollary}

\theoremstyle{definition}
\newtheorem{definition}[equation]{Definition}
\newtheorem{notation}[equation]{Notation}

\theoremstyle{remark}
\newtheorem{remark}[equation]{Remark}

\def\bC{{\mathbb C}}

\def\bN{{\mathbb N}}
\def\bP{{\mathbb P}}

\def\bR{{\mathbb R}}
\def\bZ{{\mathbb Z}}
\def\bT{{\mathbb T}}

\def\cI{{\mathcal I}}
\def\cJ{{\mathcal J}}
\def\cK{{\mathcal K}}

\def\cO{{\mathcal{O}}}
\def\cP{{\mathcal{P}}}

\def\cR{{\mathcal R}}
\def\cS{{\mathcal S}}

\newcommand\ovl[1]{\overline{#1}}

\newcommand{\FF}{\mathbb{F}}

\newcommand{\ZZ}{\mathbb{Z}}
\newcommand{\CC}{\mathbb{C}}

\newcommand{\PP}{\mathbb{P}}
\newcommand{\TT}{\mathbb{T}}

\newcommand{\RR}{\mathbb{R}}

\newcommand{\Rcal}{\mathcal{R}}

\newcommand{\Pcal}{\mathcal{P}}

\newcommand{\Oscr}{\mathcal{O}}

\newcommand{\wt}{\widetilde}

\DeclareMathOperator{\ddiv}{div}

\DeclareMathOperator{\Pic}{Pic}

\DeclareMathOperator{\id}{id}

\DeclareMathOperator{\diag}{diag}

\DeclareMathOperator{\Hom}{Hom}

\DeclareMathOperator{\Spec}{Spec}

\DeclareMathOperator{\Ab}{Ab}
\DeclareMathOperator{\Hilb}{Hilb}
\DeclareMathOperator{\Cl}{Cl}
\DeclareMathOperator{\GL}{GL}
\DeclareMathOperator{\SL}{SL}

\DeclareMathOperator{\Sp}{Sp}

\DeclareMathOperator{\codim}{codim}

\DeclareMathOperator{\Div}{Div}

\DeclareMathOperator{\lcm}{lcm}

\DeclareMathOperator{\Mon}{Mon}

\begin{document}

\title[Cox rings of symplectic resolutions of quotient singularities]{Cox rings of some symplectic resolutions\\ of quotient singularities}
\author[M.~Donten-Bury]{Maria Donten-Bury}
\address{University of Warsaw, Institute of Mathematics, Banacha 2, 02-097 Warszawa, Poland Freie Universit\"at Berlin, Mathematisches Institut, Arnimallee 3, 14195 Berlin, Germany}
\email{M.Donten@mimuw.edu.pl}

\author[M.~Grab]{Maksymilian Grab}
\address{Instytut Matematyki UW, Banacha 2, 02-097 Warszawa, Poland}
\email{M.Grab@mimuw.edu.pl}

\begin{abstract}
We investigate Cox rings of symplectic resolutions of quotients of $\mathbb{C}^n$ by finite symplectic group actions. We generalize the main result of~\cite{cox_resolutions}. We propose a finite generating set of the Cox ring of a symplectic resolution and prove that under a condition concerning monomial valuations it is sufficient. Also, three 4-dimensional examples are described in detail. Generators of the (expected) Cox rings of symplectic resolutions are computed and in one case a resolution is constructed as a GIT quotient of the spectrum of the Cox ring.
\end{abstract}

\maketitle

%%%%%-----------------------------------------------------------------------

\section{Introduction} 

This article is the second step towards understanding the total coordinate rings of resolutions of higher dimensional quotient singularities. It follows~\cite{cox_resolutions}, where the case of a certain 32-element symplectic group~$G$ is considered: first a generating set of a ring expected to be the Cox ring of a symplectic resolution of $\bC^4/G$ is given, and then it is used to construct all 81 symplectic resolutions of this singularity. The existence of symplectic resolutions of $\bC^4/G$ was proven before by Bellamy and Schedler, see~\cite{BellamySchedler2}, but no construction of such a resolution was known. Earlier, surface quotient singularities were used as a test case for developing ideas related to describing the Cox rings of resolutions and reconstructing these resolutions via Cox rings, see~\cite{MDB} and also~\cite{FGAL}. Similar concepts, now classical for toric varieties, were also studied for varieties with complexity one torus action, see~\cite{CoxRings}.

The aim of the present paper is further development of ideas presented in~\cite{cox_resolutions}. First, we study general properties of the Cox ring~$\cR(X)$ of a resolution~$X$ of a quotient of $V\simeq \bC^n$ by an action of a finite matrix group~$G$. We attempt to describe them by giving a (finite) generating set in a simpler ring: the tensor product of the Cox ring of the singularity $V/G$, which by~\cite{AG_finite} is isomorphic to the ring of invariants $\bC[V]^{[G,G]}$, and the coordinate ring of the Picard torus of the resolution. 

We restrict ourselves to symplectic singularities and their symplectic, i.e. crepant by~\cite[Thm.~2.5]{VerbitskyAsian}, resolutions. This allows us to benefit from additional information on the structure of resolutions, in particular the McKay correspondence, proven in the symplectic case by Kaledin in~\cite{KaledinMcKay}. The language of monomial valuations, used therein to relate divisors in the resolution to conjugacy classes in the group as suggested by Ito and Reid in~\cite{ItoReid}, is also present in our description of the generating set of the Cox ring.
Our main theoretical result is (cf.~\cite[Thm.~1.1]{cox_resolutions})

\begin{theorem*}[\ref{theorem:general-theorem}]
Let $G$ be a finite symplectic group generated by symplectic reflections, such that the commutator subgroup~$[G,G]$ does not contain any symplectic reflections. In addition, assume that condition~\eqref{assumption:lifting-valuations}, concerning monomial valuations on the rational function field of~$\bC[V]^{[G,G]}$, is satisfied. Let $t_1,\ldots,t_m$ denote the coordinates on the Picard torus~$\bT$ of~$X$. 
Then the Cox ring $\cR(X)$ of any symplectic resolution~$X$ of~$V/G$ is generated by
\begin{itemize}
\item $t_{1}^{E_{i}.C_{1}}\cdots t_{m}^{E_{i}.C_{m}}$ for $i = 1,\ldots, m$, and 
\item $\ovl{\phi}_{j} = \phi_{j}\cdot \prod\limits_{i=1}^m t_{i}^{\overline{D}_{j}.C_{i}}$ for $j=1,\ldots, n$,
\end{itemize}
where $\phi_j$ are eigenvectors of the action of~$Ab(G) = G/[G,G]$ on~$\bC[V]^{[G,G]}$, $\ovl{D}_j$ are strict transforms of corresponding divisors on~$V/G$, $E_i$ are exceptional divisors of the resolution and $C_i$ are curves which are generic fibers of the resolution restricted to~$E_i$.
\end{theorem*}

The general part is followed by a detailed study of three examples in dimension~4. In all of them we compute the generating set of the Cox ring proposed in Theorem~\ref{theorem:general-theorem}. The simplest one is the case of the symmetric group~$S_3$. Considered representation contains only one conjugacy class of symplectic reflections, i.e. just one monomial valuation is involved, hence condition~\eqref{assumption:lifting-valuations} is relatively easy to check. In this case there is only one symplectic resolution.

For the wreath product $\ZZ_{2}\wr S_{2}$, isomorphic to the dihedral group~$D_8$, the proof of condition~\eqref{assumption:lifting-valuations} requires an additional element. Since there are two conjugacy classes of symplectic reflections, one has to study two monomial valuations and the relation between them. In this case there are two symplectic resolution of $\bC^4/G$, which differ by a flop. We describe their geometry, in particular the structure of the central fiber, by analyzing directly the Picard torus action on the spectrum of their Cox ring and its GIT quotients.

In the last example we tackle the case of the reducible representation of the binary tetrahedral group, which will be denoted here by~$G_4$. However, the proposed set of the Cox rings generators turns out to be too complex to prove condition~\eqref{assumption:lifting-valuations} for this group.

In section~\ref{section_what_is_known} we explain how these examples fit in the general scheme of knowledge about the existence of symplectic resolution of symplectic quotient singularities.

\subsection{Outline of the paper}

In section~\ref{section:setting} we recall basic information about the Cox rings after~\cite{CoxRings} and summarize the results of section~3 of~\cite{cox_resolutions}, which gives the setting for describing the generators of Cox rings of resolutions of quotient singularities. Then, sections~\ref{subsection:candidate-generators} and~\ref{subsection:valuatins-vs-intersections} are devoted to the precise formulation of Theorem~\ref{theorem:general-theorem} together with additional assumption~\eqref{assumption:lifting-valuations}, and section~\ref{subsection:proof-of-general-thm} contains its proof.

The examples part is started by the case of $S_3$: the generators from Theorem~\ref{theorem:general-theorem} are computed in section~\ref{section_S3_generators} and in section~\ref{section_lifting_S3} we prove that they really generate the Cox ring of the resolution. The generators for $\ZZ_{2}\wr S_{2}$ are described in section~\ref{section_D8_generators} and we sketch parts of the proof that we indeed obtain the Cox ring of a symplectic resolution (a complete argument will be included in~\cite{valuationslift}). Then we investigate the geometry of the quotients of the spectrum of this ring: we embed it in an affine space in section~\ref{section_D8_ideal}, analyze stability and prove the smoothness of certain GIT quotients in sections~\ref{section_D8_stability},~\ref{section_D8_smoothness} and in section~\ref{section_D8_central_fiber} the structure of the central fiber is described. Finally, in section~\ref{section:G4} the proposed generators of the Cox ring for~$G_4$ are given.

\subsection{Existence of symplectic resolutions of quotient singularities}\label{section_what_is_known}

As proved by Verbitsky in~\cite{VerbitskyAsian}, if a symplectic resolution of~$\bC^n/G$ exists then~$G$ is generated by symplectic reflections, i.e. elements which fix a subspace of codimension~2. Note that this is a necessary but insufficient criterion. 

Finite symplectically irreducible groups generated by symplectic reflections are classified by Cohen in~\cite{Cohen}. The problem of the existence of symplectic resolutions was investigated in~\cite{GiKa, Bell, BellamySchedler, BellamySchedler2}. According to~\cite{Bell} and~\cite[4.1]{BellamySchedler2}, for the following (symplectically irreducible) $G\subset\Sp_{2n}(\CC)$ it is known that the symplectic resolution exists.
\begin{enumerate}
\item $S_{n+1}$ acting on $\CC^{2n}$ via direct sum of two copies of standard $n$-dimensional representation of $S_{n+1}$. Here the resolution can be constructed via the Hilbert scheme $\Hilb^{n+1}(\CC^{2})$. The example of~$S_3$ from section~\ref{section:S3} is the (only) 4-dimensional representative of his family.
\item $H\wr S_{n} = H^{n}\rtimes S_{n}$, where $H$ is a finite irreducible subgroup of $\SL_{2}(\CC)$, so that $\CC^{2}/H$ is one of the du Val singularities, and $S_{n}$ acts on $H^{n}$ by permutations of coordinates. The natural product representation of $H^{n}$ on $\CC^{2n}$ extend naturally to the action of $H\wr S_{n}$. The resolution is $\Hilb^{n}(S)$ where $S$ is a minimal resolution of the du Val singularity $\CC^{2}/H$. The example of $\ZZ_{2}\wr S_{2}$ analyzed in section~\ref{section:Z2wr} belongs to this family, in particular we show how to describe the resolution as a GIT quotient of the spectrum of its Cox ring.
\item A representation $G_{4}\subset\Sp_{4}(\CC)$ of the binary tetrahedral group. The existence of a symplectic resolution was first proved by Bellamy in~\cite{Bell} and constructed later by Lehn and Sorger, see~\cite{LehnSorger}. We consider this case in section~\ref{section:G4}.
\item A group of order~$32$ acting on $\CC^{4}$, for which the existence of symplectic resolutions was proved in~\cite{BellamySchedler}, and their construction and the Cox ring were investigated in~\cite{cox_resolutions}.
\end{enumerate}

Apart from this list, \cite{Bell} and~\cite{BellamySchedler2} give also negative results. The non-existence of symplectic resolutions is shown for certain infinite families of representations from the Cohen's classification~\cite{Cohen}. On the other hand, even in dimension~$4$ there still are infinitely many cases for which the question of existence of the symplectic resolution remains unanswered.

\enlargethispage{1cm}
\subsection*{Acknowledgements}

The authors are very grateful to Jaros\l{}aw Wi\'sniewski for sharing his ideas in numerous discussions and for his comments on preliminary versions of this paper.

The first author was partially supported by a Polish National Science Center project 2013/11/D/ST1/02580. This work was completed when the first author held a Dahlem Research School Postdoctoral Fellowship at Freie Universit\"at Berlin and the second author held a Doctoral Fellowship of Warsaw Center of Mathematics and Computer Science.

%%%%%-----------------------------------------------------------------------

\section{Setting}\label{section:setting}

The aim of this section is to recall the general setting for the construction of Cox rings of symplectic resolutions of quotient singularities -- we summarize the results of~\cite[Sect.~2.D, 3.A, 3.B]{cox_resolutions}.

Let $V =\bC^{2n}$. We consider quotients $Y = V/G$, where $G\subset\Sp(V)$ is a finite subgroup of linear transformations of~$V$ preserving a symplectic form on~$V$. 
Assume that a symplectic resolution of singularities $$\varphi \colon X\to Y$$ exists. Our goal is to determine the generators of the Cox ring of~$X$, and then to use this ring to recover the construction of the resolution. 

Recall that the Cox ring, or the total coordinate ring, of a normal, irreducible complex algebraic variety $Z$ is the module
$$\cR(Z)=\bigoplus_{[D]\in\Cl(Z)}\Gamma(Z,\cO_Z(D))$$
with multiplication coming from the identification 
$$\Gamma(Z,\cO_Z(D))= \{f\in\bC(Z)^*: \ddiv(f)+D\geq 0\}\cup\{0\}$$
of global sections with rational functions on~$Z$. If the class group of Weil divisors $\Cl(Z)$ is finitely generated and free, as by~\cite[Lem.~2.13]{cox_resolutions} it happens in the case of resolutions of quotient singularities, it is enough to fix inclusions $\Gamma(Z,\cO_Z(D_i)) \subset \bC(Z)$ for chosen representatives $D_i$ of a basis $\{[D_i] \colon i=1,\ldots,n\}$ of $\Cl(Z)$, and multiply sections as corresponding rational functions. Otherwise, relations in $\Cl(Z)$ have to be taken into account, for the details see~\cite[Sect.~1.4]{CoxRings}.

Assume that $\Cl(Z)$ is finitely generated and free. If $\cR(Z)$ is finitely generated $\bC$-algebra (thus $Z$ is a Mori Dream Space), there is an action of the torus $\bT_{\Cl(Z)} = \Hom(\Cl(Z),\bC^*)$ on its spectrum, coming from the $\Cl(Z)$ grading. If in addition $Z$ is projective over $\bC[Z]$, we may recover $Z$, and its other birational models, as GIT quotients of $\Spec \cR(Z)$ by $\bT_{\Cl(Z)}$. This is the main idea behind resolving quotient singularities via Cox rings: if $\cR(Z)$ can be found just based on the group structure data and general information on the resolution, one may try to construct new resolutions as GIT quotients of $\Spec \cR (Z)$.

\subsection{Embedding of the Cox ring of a resolution}

 By $\cR(X)$ and $\cR(Y)$ we denote the Cox rings of the resolution~$X$ and the singular quotient~$Y$ respectively. One can describe the class groups of~$X$ and~$Y$, by which these rings are graded.
\begin{proposition}
By~\cite[Prop.~3.9.3]{Benson} and~\cite[Lem.~2.13, 2.14]{cox_resolutions} we have:
\begin{enumerate}[label=(\alph*)]
\item $\Cl(Y)$ is naturally isomorphic to the group $G^{\vee} = \Hom(G,\CC^{*})=\Ab(G)^{\vee}$,
\item $\Cl(X) = \Pic(X) \simeq \bZ^m$, where $m$ is the number of irreducible components of the exceptional divisor, which by the McKay correspondence (see~\cite{KaledinMcKay}) is equal to the number of conjugacy classes of symplectic reflections in $G$.
\end{enumerate}
\end{proposition}   

We have a natural surjective map $\varphi_{*}\colon \Cl(X)\to \Cl(Y)$ induced by the push-forward of Weil divisors via~$\varphi$. The resolution $\varphi$ induces also a natural injective homomorphism $\Gamma(X,\Oscr_{X}(D))\to \Gamma(Y,\Oscr_{Y}(\varphi_{*}D))$ of modules of global sections of rank 1 reflexive sheaves associated with a Weil divisor $D$ on $X$ and its push-forward $\varphi_*(D)$, see~\cite[3.A]{cox_resolutions}. Altogether we obtain a homomorphism of graded rings
\begin{equation*}
\varphi_{*}\colon \Rcal(X)\to\Rcal(Y).
\end{equation*}

There is also a comultiplication homomorphism
$$\Rcal(X)\to \Rcal(X)\otimes \CC[\Cl(X)],$$
given by the action of the Picard torus $\TT_{\Cl(X)}$ on $\Spec \Rcal(X)$, compatible with the $\Cl(X)$-grading. On the level of graded pieces of $\Rcal(X)$ it is given by
$$\Gamma(X,\Oscr_{X}(D))\ni f \mapsto f\otimes \chi^{D}\in \Gamma(X,\Oscr_{X}(D))\otimes \CC[\Cl(X)],$$
where $\chi^{D}$ is a character of torus $\TT_{\Cl(X)}$ corresponding to the class of $D$ in $\Cl(X)$.

\begin{proposition}[{\cite[Prop.~3.8]{cox_resolutions}}]\label{proposition:cox-rings-embedding}
The composition of the comultiplication homomorphism and $\varphi_{*}\otimes \id$
\begin{equation*}
\Theta:\Rcal(X)\to \Rcal(X)\otimes \CC[\Cl(X)]\to \Rcal(Y)\otimes \CC[\Cl(X)]
\end{equation*}
is an injective graded ring homomorphism.
\end{proposition}

By~\cite[Thm.~3.1]{AG_finite} (see also~\cite[Lem.~6.2]{MDB}) the Cox ring $\Rcal(Y)$ can be viewed as the ring of invariants $\CC[V]^{[G,G]}\subset \CC[V]$ of commutator subgroup $[G,G]$, with grading given by the induced action of $\Ab(G)$ on $V/[G,G]$. Hence by Prop.~\ref{proposition:cox-rings-embedding} we may identify $\Rcal(X)$ with a subring of a better understood ring: 
\begin{equation}
\Rcal(X) \subseteq \CC[V]^{[G,G]}\otimes \CC[\Cl(X)].
\end{equation}

Thus, the task of determining $\Rcal(X)$ is reduced to pointing out a set of elements of $\CC[V]^{[G,G]}\otimes \CC[\Cl(X)]$ which generate $\Rcal(X)$ as a $\bC$-algebra.

\subsection{Describing the image of $\Theta$}
\label{subsection:image-of-theta}

Let $E_{1},\ldots, E_{m}$ be the components of exceptional divisor of the resolution $\varphi \colon X \to Y$. Let $C_{i}$ denote the generic fiber of $\varphi|_{E_{i}}\colon E_{i}\to \varphi(E_{i})$. As it was noted in~\cite[Sect.~2.D]{cox_resolutions}, classes of $C_{1},\ldots, C_{m}$ form a basis of the vector space $N_{1}(X/Y)$. The dual basis of $N^{1}(X/Y)$ (via the intersection pairing) will be denoted by $L_{1},\ldots, L_{m}$. Then the coefficients of a divisor~$D$ on~$X$ in this basis are the intersection numbers $(C_{i}.D)$.

\begin{proposition}[{\cite[2.16]{cox_resolutions}}]\label{prop_diag_class_groups}
We have a commutative diagram, whose rows are exact sequences:
\begin{equation}\label{diagram:class-groups}
\xymatrix{
0 \ar[r] & \bigoplus\limits_{i=1}^{m}\ZZ E_{i} \ar[r]\ar@{=}[d] & \Cl(X) \ar[r]^{\varphi_{*}} \ar@{^{(}->}[d] & \Cl(Y) \ar[r] \ar@{^{(}->}[d] & 0 \\
0 \ar[r] & \bigoplus\limits_{i=1}^{m}\ZZ E_{i} \ar[r] & \bigoplus\limits_{i=1}^{m}\ZZ L_{i} \ar[r] & Q \ar[r] & 0,
}
\end{equation}
Here the homomorphism $\bigoplus_i\ZZ E_{i} \to \bigoplus_i\ZZ L_{i}$ takes $E_{i}$ to $\sum_{i}(E_{i}.C_{j})L_{j}$, and the group $Q$ is defined as its cokernel. The image of $ D \in\Cl(Y)$ in $Q$ is given by $D\mapsto \sum_{i}(\varphi^{-1}_{*}(D).C_{i})[L_{i}]$. 
\end{proposition}

It is known that variety $V/G$ has du Val singularities in codimension $2$ and the intersection matrix $(E_{i}.C_{j})$ is a direct sum of corresponding Cartan matrices, see~\cite[Thm.~1.3]{WierzbaWisniewski} and~\cite[Thm.~1.4]{AW}. In particular, $(E_{i}.C_{j})$ is invertible and so the lattice $\Cl(X) = \Pic(X) \subset N^{1}(X/Y)$ is a finite index sublattice of the lattice $\bigoplus_{i=1}^{m} \ZZ L_{i}\subset N^{1}(X/Y)$.

\begin{remark}\label{remark:class-group-sublattice} Note that by Proposition~\ref{prop_diag_class_groups} we have $\Cl(X) = \bigoplus_i\ZZ L_{i}$ if and only if the index of subgroup $\bigoplus_{i}\ZZ E_{i}\subseteq \bigoplus_i\ZZ L_{i}$ is equal to the order of $\Cl(Y)\cong \Ab(G)=G/[G,G]$. The index of $\bigoplus_{i}\ZZ E_{i}\subseteq \bigoplus_i\ZZ L_{i}$ is the absolute value of the determinant of the intersection matrix $(E_{i}.C_{j})$. In the case of the 32-element group considered in~\cite{cox_resolutions} the order of $\Ab(G)$ is smaller than $|\det (E_{i}.C_{j})_{i,j}|$: there $\Ab(G) \cong \ZZ^{4} $ but $(E_{i}.C_{j})_{i,j}$ is the direct sum of 5 copies of $(-2)$. However, it is the only case in dimension $4$ where it is known that a symplectic resolution exists and 
$|\Ab(G)| < |\det(E_{i}.C_{j})|$. This can be verified via direct inspection of 4-dimensional cases from the list in section~\ref{section_what_is_known}.
\begin{itemize}
\item For groups of the form $H\wr S_{2}$, where $H\subset \SL_{2}(\CC)$, the intersection matrix is a direct product of two Cartan matrices: one of $A_1$ type, i.e. $(-2)$, and $C_{H}$ corresponding to the singularity $\CC^{2}/H$. We have $\Ab(H\wr S_{2}) = \Ab(H)\times S_{2}$. The claim follows from the observation that $|\Ab(H)| = |\det C_H|$.
\item For the group $S_{3}$ we have $|\Ab(G)|=2$ and the corresponding Cartan matrix is $(-2)$.
\item For the group $G_{4}$ we have $|\Ab(G)| = 3$ and $(E_{i}.C_{j})$ is $A_{2}$-type Cartan matrix, so $\det(E_{i}.C_{j}) = 3$.
\end{itemize}

In particular, in examples given in sections~\ref{section:S3}-\ref{section:G4} the diagram~\eqref{diagram:class-groups} reduces to one row.
\end{remark}

If we embed $\CC[\Cl(X)]$ into $\CC\left[\bigoplus_i\ZZ L_{i}\right]$ and denote by $t_i$ the variable corresponding to $[L_i]$, the homomorphism $\Theta$ introduced in Proposition~\ref{proposition:cox-rings-embedding} becomes the injective graded ring homomorphism
\begin{equation}
\overline{\Theta}:\Rcal(X)\to \Rcal(Y)[t_{1}^{\pm 1},\ldots, t_{m}^{\pm 1}].
\end{equation}
In particular, we may identify $\Rcal(X)$ with its image $\overline{\Theta}(\Rcal(X))$. 

Note that by Remark~\ref{remark:class-group-sublattice} in cases $S_{3}$, $\ZZ_{2}\wr S_{2}$, $G_{4}$ (see sections~\ref{section:S3}-\ref{section:G4}) we have $\Cl(X) = \bigoplus_i\ZZ L_{i}$, that is $\overline{\Theta}=\Theta$. 

To describe the elements of the image of $\overline{\Theta}$ we will use the following convention (see~\cite[Sect.~2.A]{cox_resolutions}).

\begin{notation}\label{notation_corresp_section}
Let $D$ be a Weil divisor on a variety $Z$ and let $\Oscr(D)$ be the associated reflexive sheaf. For each section in $\Gamma(V, \Oscr(D)) = \{f\in K(Z)^{*}\ : \ \ddiv f + D \ge 0\}\cup \{0\}$ one may take the corresponding effective divisor $D' = \ddiv f + D$. We denote the section corresponding to the divisor $D'$ by $f_{D'}$, it is determined up to the multiplication by a constant. If $Z'\to Z$ is a birational map then we will denote the strict transform of $D$ by $\ovl{D}$.
\end{notation}

The next result, cf.~\cite[Cor.~3.11]{cox_resolutions}, follows directly from the definition of $\overline{\Theta}$.
\begin{proposition}\label{proposition:image-via-theta}
If $D\in \Div(X)$ and $f_{D}\in \Rcal(X)$, then 
\begin{equation}
\overline{\Theta}(f_{D}) = f_{\varphi_{*}D}\cdot t_{1}^{D.C_{1}}\cdots t_{m}^{D.C_{m}}.
\end{equation}
\end{proposition}

\subsection{Monomial valuations}
\label{subsection:monomial-valuations}
We will need the concept of monomial valuations, which are used e.g. in~\cite{KaledinMcKay} in the context of the McKay correspondence for resolutions of quotient symplectic singularities. For each positive integer~$r$ fix a root of unity $\epsilon_{r}\in\CC^{*}$ of order~$r$. Take a group monomorphism $\mu_{r} = \langle \epsilon_{r}\rangle \to \GL_{n}(\CC)$. This is equivalent to fixing a linear operator 
$T\in \GL_{n}(\CC)$ of order $r$, the image of the chosen generator $\epsilon_{r}$. Assume that $T$ acts on $\CC^{n}$ with coordinates $x_{1},\ldots,x_{n}$ by a diagonal matrix $\diag(\epsilon_{r}^{a_{1}},\ldots, \epsilon_{r}^{a_{n}})$ where $0\le a_{i}<r$.
\begin{definition}
The monomial valuation $\nu_{T}:\CC(x_{1},\ldots,x_{n})\to \ZZ$ is defined by
\begin{equation*}
\nu_{T}\left(\sum_{\alpha} c_{\alpha}x^{\alpha}\right)= \min\{\langle(a_{1},\ldots,a_{n}), \alpha\rangle\ : \ c_{\alpha}\neq 0\},
\end{equation*}
where $\langle\alpha,\beta\rangle = \sum_{i}\alpha_{i}\beta_{i}$.
\end{definition}

By the McKay correspondence each component $E_{i}$ of the exceptional divisor of the resolution $\varphi$ corresponds to a conjugacy class of symplectic reflections. We fix representatives of conjugacy classes: $T_{i}$ represents the $i$-th class. Let~$r_{i}$ be the order of~$T_{i}$. 
We will write $\nu_{i}$ for $\nu_{T_{i}}$. The next lemma will be useful in relating values of $\nu_{i}$ to the intersection numbers of divisors with curves $C_{i}$.

\begin{lemma}[{\cite[Lem.~2.5]{KaledinMcKay}}]\label{lemma:restricting-monomial-valuation}
$\nu_{i}|_{\CC(V)^{G}} = r_{i}\nu_{E_{i}}$, where $\nu_{E_{i}}$ is the divisorial valuation centered at $E_i$. 
\end{lemma}

%%%%%-----------------------------------------------------------------------

\section{General results}
\label{section:general-proof}

Here we explain our strategy for finding generators of the Cox ring $\Rcal(X)$. First, in~\ref{subsection:candidate-generators} we propose the generating set and introduce the technical assumption on monomial valuations under which Theorem~\ref{theorem:general-theorem} is then proved in~\ref{subsection:proof-of-general-thm}. The proof relies on the relation between monomial valuations and intersections with curves on symplectic resolution, described in~\ref{subsection:valuatins-vs-intersections}.

\begin{remark}
Throughout this section we will always assume that the commutator subgroup $[G,G]$ does not contain symplectic reflections.  This guarantees that orders of symplectic reflections in $G$ are the same as orders of their images in $\Ab(G)$. 
\end{remark}

\begin{remark}\label{remark:An-type-singularities} As we noted in~\ref{subsection:image-of-theta}, the symplectic quotient singularity $V/G$ have codimension $2$ du Val singularities. However, under the assumption that $[G,G]$ does not contain symplectic reflections $V/G$ can have only codimension $2$ singularities of $A_{n}$-type. Indeed, each of those singularities will occur along the image of subspace fixed by some symplectic reflection. Then the symplectic reflections preserving this subspace form naturally a subgroup of $\SL_{2}(\CC)$. The commutator of this subgroup does not contain symplectic reflections if and only if it is trivial, i.e. this subgroup is cyclic.
\end{remark}

\subsection{A candidate for a generating set}
\label{subsection:candidate-generators}
To simplify the notation for the Cox ring of $Y = V/G$ we set $\Pcal =\Rcal(Y) \simeq \CC[V]^{[G,G]}$. 
We look at the $\Cl(Y)$-grading on~$\Pcal$, i.e. the decomposition into eigenspaces of the $\Ab(G)$ action, and fix a finite generating set of $\cP$.

\begin{notation} 
By $\{\phi_{1},\ldots, \phi_{n}\}$ we denote a finite set of generators of $\cP$ consisting of chosen eigenvectors of the $\Ab(G)$ action. 
We think of each $\phi_j$ as of a global section of a reflexive sheaf on $Y$ and denote its divisor of zeroes by~$D_{j}$ and their strict transforms by~$\ovl{D}_j$. 
\end{notation}

We lift elements $\phi_j \in \cP$ to $\cP[t_{1}^{\pm 1},\ldots, t_{m}^{\pm 1}]$ as follows: 
$$\ovl{\phi}_{j} = \phi_{j}\cdot \prod\limits_{i=1}^m t_{i}^{\overline{D}_{j}.C_{i}}.$$ 

The aim of this section is to prove that in some situations elements $\ovl{\phi}_j$ together with certain characters of the Picard torus form a (finite) generating set of $\cR(X)$, see Theorem~\ref{theorem:general-theorem}. However, we need first to describe an additional assumption under which the result holds.

Let $\kappa\colon \CC[w_{1},\ldots, w_{n}]\to \Pcal$ be the surjective ring homomorphism sending $w_{i}\mapsto \phi_{i}$. We will need to compare monomial valuations~$\nu_i$ on~$\cP$ corresponding to symplectic reflections~$T_i$ with certain monomial valuations on $\CC[w_{1},\ldots, w_{n}]$. To define them we lift each $T_i$ to a linear map $\wt{T_i} \colon \bC^n \to \bC^n$
\begin{equation*}
\wt{T}_i(w_{1},\ldots,w_n) = \left(\frac{T_{i}(\phi_{1})}{\phi_{1}} w_{1},\ldots, \frac{T_{i}(\phi_n)}{\phi_n} w_n\right),
\end{equation*}
which is well-defined since $\phi_j$ are eigenvectors of~$\Ab(G)$.

\begin{definition}\label{def-valuation-up}
By~$\wt{\nu}_{i}$ we understand the monomial valuation corresponding to~$\wt{T}_{i}$. 
\end{definition}

Let~$r_i$ be the order of the image of symplectic reflection $T_{i}$, corresponding to exceptional divisor $E_i$, in $G$. Since $[G,G]$ does not contain symplectic reflections, $r_{i}$ is also the order of the class of $T_{i}$ in $\Ab(G)$. Note that $\wt{T}_{i}$ also has order $r_{i}$.

\begin{remark}\label{remark:valuations-of-generators}
Let~$\epsilon_{r_{i}}$ be the $r_i$-th root of unity, fixed in the definition of monomial valuations in section~\ref{subsection:monomial-valuations}. Then $\frac{T_{i}(\phi_{j})}{\phi_{j}} = \epsilon_{r}^{a_{ij}}$ where $0\le a_{ij}< r_{i}$, that is $\wt{\nu}_{i}(w_{j}) = a_{ij}$. By the definition of a monomial valuation we see that $\nu_{i}(\phi_{j}) \ge 0$, $\nu_{i}(\phi_{j}) \equiv a_{ij}\pmod{r_{i}}$. Hence, in particular, $\nu_{i}(\phi_{j})\ge \wt{\nu}_{i}(w_{j})$.
\end{remark}

In the proof of Theorem~\ref{theorem:general-theorem} we assume that every homogeneous $f\in \Pcal$ can be lifted through $\kappa$ compatibly with valuations $\nu_i$ and $\wt{\nu}_i$. More precisely,

\refstepcounter{equation}\label{assumption:lifting-valuations}
\begin{flushleft}
\vspace{0.1cm}
\begin{tabular}{lm{11cm}}
  (\theequation) & \emph{for any homogeneous $f\in \Pcal$ there is $\wt{f}\in \CC[w_{1},\ldots,w_n]$ such that $\kappa(\wt{f}) = f$ and $\nu_{i}(f) \le \wt{\nu}_{i}(\wt{f})$ for all $i$.} \\
\end{tabular}
\vspace{0.1cm}
\end{flushleft}
Note that the other inequality is automatic, it follows directly from the construction of~$\wt{\nu}_i$ and the definition of monomial valuations.

\begin{remark}
The motivation for introducing this condition comes from the embedding of the singularity and its resolution in toric varieties, see~\cite[Sect.~4.2]{MDB} for the surface case. The ring $\CC[w_{1},\ldots,w_n]$ is naturally identified with the Cox ring of the toric quotient $\bC^n/Ab(G)$. This quotient is singular along fixed subspaces of $\wt{T}_i$ and one may construct a partial toric resolution, smooth over generic point of these subspaces. Its Cox ring can be described based on the toric data, and it is expected to have surjective homomorphism to $\cR$. Thus the idea of proving the inclusion $\cR(X) \subseteq \cR$ is, roughly speaking,  as follows: take an element of $\cR(X)$, map it to $\cP$, lift to $\cR(\bC^n/Ab(G)) \simeq \CC[w_{1},\ldots,w_n]$, and use the description of the Cox ring of the partial toric resolution of $\bC^n/Ab(G)$ to show a presentation in terms of generators of $\cR$. Here condition~\eqref{assumption:lifting-valuations} is necessary to ensure that in this process we obtain elements of suitable graded pieces of~$\cR$. However, it turns out that a direct reference to toric arguments can be avoided, see the proof of Proposition~\ref{proposition:strict-transforms}.
\end{remark}

Now we can formulate the main result of this section.

\begin{theorem}\label{theorem:general-theorem}
Suppose that $G$ is such that $[G,G]$ does not contain symplectic reflections and assume that condition~\eqref{assumption:lifting-valuations} holds. Then $\ovl{\Theta}(\cR(X))$ is generated by
\begin{itemize}
\item $t_{1}^{E_{i}.C_{1}}\cdots t_{m}^{E_{i}.C_{m}}$ for $i = 1,\ldots, m$, and 
\item $\ovl{\phi}_{j} = \phi_{j}\cdot \prod\limits_{i=1}^m t_{i}^{\overline{D}_{j}.C_{i}}$ for $j=1,\ldots, n$.
\end{itemize} 
\end{theorem}

The proof requires some preparation and will be given in section~\ref{subsection:proof-of-general-thm}.

\begin{remark}\label{remark:philosophy}
Every effective divisor on $X$ can be written as $\overline{D} + \sum_{i}a_{i}E_{i}$ with $a_{i}$ non-negative integers and $\overline{D}$ being a strict transform of an effective Weil divisor $D$ on $V/G$. In particular the class group $\Cl(X) = \Pic(X)$ is generated by classes of effective divisors $\overline{D}$ together with components $E_{i}$ of the exceptional divisor.

By proposition~\ref{proposition:image-via-theta} elements of the form $t_{1}^{E_{i}.C_{1}}\cdots t_{m}^{E_{i}.C_{m}}$ are images of sections of line bundles $\Oscr_{X}(E_{i})$, and the elements $\phi_{j}\cdot t_{1}^{\overline{D}_{j}.C_{1}}\cdots t_{m}^{\overline{D}_{j}.C_{m}}$ are images of sections of line bundles $\Oscr_{X}(\overline{D}_{j})$. In consequence, the generating set of $\cR(X)$ proposed in Theorem~\ref{theorem:general-theorem} allows us to obtain elements of all graded pieces. 
\end{remark}

Condition~\eqref{assumption:lifting-valuations} is in general not easy to show. In section~\ref{section_lifting_S3} it is verified in the case $G = S_{3}$. In section~\ref{section_D8_lifting} some elements of the proof for $G = \ZZ_{2}\wr S_{2}$ are explained, and a complete argument will be given in~\cite{valuationslift}.

\subsection{Describing intersections in terms of monomial valuations}
\label{subsection:valuatins-vs-intersections}
In the proof of Theorem~\ref{theorem:general-theorem} we will use the language of monomial valuations, see~\ref{subsection:monomial-valuations}.
Here we establish the relation between monomial valuations corresponding to symplectic reflections in $G$ and intersection numbers of divisors on $X$ with curves $C_{1},\ldots, C_{m}$.

Let $D$ be a Weil divisor on $Y$ and $f_{D}\in \Pcal$ be the corresponding section, see Notation~\ref{notation_corresp_section}. We express the intersection numbers $\overline{D}.C_{i}$ by monomial valuations $\nu_{i}(f_{D})$. This is used later to give a different description of proposed generators of~$\ovl{\Theta}(\cR(X))$.

Set $r=\lcm \{r_{1},\ldots, r_{m}\}$. Then $f_{D}^{r}\in \CC[V]^{G}\subset \CC(V)^{G} = \CC(X)$ and we can take the associated principal divisor on $X$: 
\begin{equation*}
0\sim \ddiv_{X} f_{D}^{r} = r\overline{D} + \sum_{i}r\nu_{E_{i}}(f_{D})E_{i}.
\end{equation*}

 Using Lemma~\ref{lemma:restricting-monomial-valuation} and intersecting with $C_{j}$ we get 
\begin{equation}\label{eqn:intersections-vs-valuations}
r\overline{D}.C_{j} = - \sum_{i}\frac{r}{r_{i}}\nu_{i}(f_{D})\cdot(E_{i}.C_{j}).
\end{equation}

\begin{remark}
The intersection matrix $(E_{i}.C_{j})_{i,j}$ is invertible as a direct sum of Cartan matrices. If $U$ is its inverse then we may rewrite~\eqref{eqn:intersections-vs-valuations} as
\begin{equation}\label{eqn:cartan_matrix}
\left(\frac{1}{r_{1}}\nu_{1}(f_{D}),\ldots, \frac{1}{r_{m}}\nu_{m}(f_{D})\right)=- (\ovl{D}.C_{1},\ldots, \ovl{D}.C_{m})\cdot U.
\end{equation}
Hence for the strict transforms of Weil divisors on $V/G$ the information encoded by monomial valuations $\nu_{i}$ and by the intersections with curves $C_{i}$ is the same. 
\end{remark}

\subsection{The image of the Cox ring is generated by indicated elements}
\label{subsection:proof-of-general-thm}
We denote the set consisting of elements listed in Theorem~\ref{theorem:general-theorem} by $\cS$. Let~$\cR$ be the subring of $\Pcal[t_{1}^{\pm 1},\ldots, t_{m}^{\pm 1}]$ generated by $\cS$. 
The nontrivial part of Theorem~\ref{theorem:general-theorem} is to show that elements of $\cS$ are sufficient to generate $\ovl{\Theta}(\Rcal(X))$.

The next proposition is the major step toward proving Theorem~\ref{theorem:general-theorem}. In particular, it implies that $\cS$ is sufficient to generate sections of all line bundles $\Oscr_{X}(\overline{D})$ for Weil divisors~$D$ on~$V/G$, see also Remark~\ref{remark:philosophy}.

\begin{proposition}\label{proposition:strict-transforms}
Assume that condition~\eqref{assumption:lifting-valuations} holds. Let $f = f_{D}\in \Pcal$ be a homogeneous element with zero-divisor~$D$. Then $f_{D}\cdot \prod_i t_{i}^{\overline{D}.C_{i}} \in \Rcal$.
\end{proposition}

\begin{proof}
Let $M = (\ovl{D}_i.C_j)_{i,j}$ for $i = 1,\ldots,n$, $j = 1,\ldots,m$ be the intersection matrix of strict transforms of divisors corresponding to $\phi_i$ with curves $C_j$. By Proposition~\ref{proposition:image-via-theta} this is the (transposed) matrix of weights of the Picard torus action on $\phi_i$. Also, $I = (i_1,\ldots,i_m) \in \bZ^m$ and $K= (k_1,\ldots,k_n) \in \bZ^n$ will be multiindices.

Take a lift $\wt{f} = \wt{f}_{D}\in \CC[w_{1},\ldots, w_n]$ of $f$ through $\kappa$ compatible with valuations, which exists by condition~\eqref{assumption:lifting-valuations}.
Set $F = \wt{f}(\overline{\phi}_{1}, \ldots, \overline{\phi}_n).$ 

We divide $\wt{f}$ into parts which in $F$ go to elements homogeneous in variables $t_1,\ldots,t_m$. That is, $\wt{f} = \sum_I\ovl{\beta}_I$, where 
$$\ovl{\beta}_I = \sum_{\{K \colon KM = I\}} \alpha_K\cdot w_1^{k_1}\cdots w_n^{k_n}.$$

Let $\beta_I = \kappa(\ovl{\beta}_I) = \ovl{\beta}_I(\phi_{1},\ldots, \phi_n)$. Note that
$\beta_I \prod_j t_{j}^{i_{j}} = \ovl{\beta}_I \left(\ovl{\phi}_{1}, \ldots, \ovl{\phi}_n\right)$
is an element of $\cR$.
Now we may write
$$f\cdot \prod_j t_{j}^{\overline{D}.C_{j}} = \kappa(\wt{f})\cdot \prod_j t_{j}^{\overline{D}.C_{j}} = \left(\sum_I \kappa(\ovl{\beta}_I)\right)\cdot \prod_j t_{j}^{\overline{D}.C_{j}} =\sum_I \left(\beta_I \prod_j t_{j}^{i_{j}}\right)\cdot \prod_j t_{j}^{\ovl{D}.C_j - i_{j}}.$$

To show that the right hand side is indeed in $\cR$ it suffices to prove that for any multiindex $I$ such that $\beta_I \neq 0$ we have $\prod_j t_{j}^{\ovl{D}.C_j - i_{j}} \in \cR$.

We will show how to get non-negative integers $a_1,\ldots,a_m$ such that
$$\prod\limits_{j=1}^m t_{j}^{\ovl{D}.C_{j} - i_{j}}= \left(\prod_{i=1}^nt_{i}^{E_{1}.C_{i}}\right)^{a_{1}}\cdots \left(\prod_{i=1}^nt_{i}^{E_{m}.C_{i}}\right)^{a_{m}}.$$

This will end the proof since factors of the right hand side are in the chosen generating set~$\cS$ of~$\cR$.

This means that $(a_1,\ldots,a_m) = (\ovl{D}.C_{1} - i_{1}, \ldots,\ovl{D}.C_{m} - i_{m})\cdot U$, where $U$ is the inverse of the Cartan matrix for considered singularity. By~\eqref{eqn:cartan_matrix} applied to $\ovl{D}$ we get $(\ovl{D}.C_1,\ldots,\ovl{D}.C_m)\cdot U = -\left(\frac{1}{r_{1}}\nu_{1}(f),\ldots,\frac{1}{r_{m}}\nu_{m}(f)\right)$. 

Take any $K = (k_1,\ldots,k_m)$ such that $KM = I$ and $\alpha_K \neq 0$.
Now~\eqref{eqn:cartan_matrix} applied to $\sum k_i\ovl{D}_i$ gives $I U = KMU = -\left(\frac{1}{r_{i}}\nu_{1}(\phi_1^{k_1}\cdots \phi_n^{k_n}),\ldots,\frac{1}{r_{m}}\nu_{m}(\phi_1^{k_1}\cdots \phi_n^{k_n})\right)$.

That is, we may take $a_j = \frac{1}{r_{j}}(\nu_{j}(\phi_1^{k_1}\cdots \phi_n^{k_n}) - \nu_{j}(f))$, as soon as we show that they are non-negative integers.
For the non-negativity observe that
\begin{equation*}
\nu_{j}(f) \le \wt{\nu}_{j}(\wt{f}) \le \wt{\nu}_{j}(w_{1}^{k_{1}}\cdots w_{n}^{k_{n}}) \le \nu_{j}(\phi_{1}^{k_{1}}\cdots\phi_{n}^{k_{n}}),
\end{equation*}
where the first inequality is a consequence of assumption~\eqref{assumption:lifting-valuations}, the second follows from the definition of monomial valuations and the last one from Remark~\ref{remark:valuations-of-generators}.

It remains to prove that $a_{j}$ are integers. By assumption~$f$ is homogeneous with respect to $\Cl(Y)$-grading, and since $\phi_i$ are also homogeneous we may assume that $\phi_{1}^{k_1}\cdots\phi_{n}^{k_n}$ has the same degree as $f$. Indeed, images of summands $\alpha_{K'}\cdot w_{1}^{k'_1}\cdots w_{n}^{k'_n}$ mapped by $\kappa$ to elements of different degrees must cancel, and removing them from $\wt{f}$ does not decrease $\wt{\nu}_{i}(\wt{f})$. 
But then $\frac{\phi_{1}^{k_1}\cdots \phi_{n}^{k_n}}{f}\in \CC(V)^{G}$ and by Lemma~\ref{lemma:restricting-monomial-valuation}
$$a_{j} = \frac{1}{r_{j}} \nu_{j}\left(\frac{\phi_{1}^{k_1}\cdots \phi_{n}^{k_n}}{f}\right) = \nu_{E_j} \left(\frac{\phi_{1}^{k_1}\cdots \phi_{n}^{k_n}}{f}\right) \in \ZZ.$$

\end{proof}

\begin{proof}[Proof of theorem~\ref{theorem:general-theorem}]
By Proposition~\ref{proposition:image-via-theta} (see also Remark~\ref{remark:philosophy}) we have $\ovl{\Theta}(\cR(X)) \supset \cR$, hence it suffices to show $\ovl{\Theta}(\Rcal(X))\subset \Rcal$.

We will use Proposition~\ref{proposition:strict-transforms}. Since $\ovl{\Theta}$ is a graded ring homomorphism, it is sufficient to show that the images of homogeneous elements via $\ovl{\Theta}$ are all in $\Rcal$. Take $f = f_{D}\in \Rcal(X)$. We have to show that $\ovl{\Theta}(f) = f_{\varphi_{*}D}\cdot \prod_{i}t_{i}^{D.C_{i}}\in \Rcal$. We may write $D = \overline{\varphi_{*}D} + \sum_{j}a_{j}E_{j}$ and we obtain
\begin{equation*}
f_{\varphi_{*}D}\cdot \prod_{i}t_{i}^{D.C_{i}}=f_{\varphi_{*}D}\cdot \prod_{i}t_{i}^{\overline{\varphi_{*}D}.C_{i}}\cdot \prod_{j}\left(\prod_{i}t_{i}^{E_{j}.C_{i}}\right)^{a_{j}}.
\end{equation*}
This is an element of $\cR$, because $f_{\varphi_{*}D}\cdot \prod_{i}t_{i}^{\overline{\varphi_{*}D}.C_{i}}\in \Rcal$ by Proposition~\ref{proposition:strict-transforms} and the remaining factors are expressed in terms of the chosen generating set $\cS$ of~$\Rcal$.
\end{proof}

%%%%%-----------------------------------------------------------------------

\section{The symmetric group $S_{3}$}
\label{section:S3}
\subsection{Generators of the Cox ring}\label{section_S3_generators}

In this section we investigate an action of $G = S_{3}$ on $V \simeq \bC^4$ with coordinates $(x_{1},y_{1},x_{2},y_{2})$. We take the direct sum of two copies of the standard two-dimensional representation of $S_3$, which preserves the symplectic form $\omega = dx_{1}\wedge dy_{2} + dy_{1}\wedge dx_{2}$ on $V$. That is, we identify~$G$ with the matrix group generated by
\begin{equation*}
\varepsilon =
\begin{pmatrix}
\epsilon & 0 & 0 & 0\\
0 & \epsilon^{-1} & 0 & 0\\
0 & 0 & \epsilon & 0\\
0 & 0 & 0 & \epsilon^{-1}
\end{pmatrix} \qquad T =
\begin{pmatrix}
0 & 1 & 0 & 0\\
1 & 0 & 0 & 0\\
0 & 0 & 0 & 1\\
0 & 0 & 1 & 0
\end{pmatrix}
\end{equation*}
where $\epsilon$ is a $3$rd root of unity. Let $\nu$ be the monomial valuation $\CC(V)\to \ZZ$ associated with symplectic reflection $T$. The following proposition summarizes basic properties of this representation; all of them are immediate.

\begin{proposition} With the notation above:
\begin{itemize}
\item Commutator subgroup $H = [G,G]$ is a cyclic group of order $3$ generated by $\varepsilon$. In particular $H$ does not contain any symplectic reflection.
\item $\Ab(G) = G/H = \langle T H\rangle \cong \ZZ_{2}$.
\item Symplectic reflections in $G$ are $\varepsilon^{i} T$ for $i =0,1,2$. These elements correspond to transpositions in $S_{3}$ and as such generate~$G$. 
\item All those symplectic reflections are conjugate. In particular, by the McKay correspondence, the symplectic resolution $\varphi\colon X\to Y =V/G$ has irreducible exceptional divisor~$E$.
\item The intersection matrix is the Cartan matrix of type $A_{1}$. In other words, if $C$ is a generic fiber of $\varphi|_{X}:E\to \varphi(E)$, then we have $E.C = -2$.
\item $\nu|_{\CC(V)^{G}} = 2\nu_{E}$ by Lemma~\ref{lemma:restricting-monomial-valuation}.
\end{itemize}
\end{proposition}

\begin{table}[h!]
\begin{tabular}{c l}
eigenvalue & generators  \\
$1$ 	& $\phi_{1}= x_{1}y_{1},\quad \phi_{2} = x_{2}y_{2},\quad \phi_{3} = x_{1}y_{2} + x_{2}y_{1}$ \\
        & $\phi_{4} = x_{1}^{3} + y_{1}^{3}, \ \phi_{5} = x_{2}^{3} + y_{2}^{3},\ \phi_{6} =x_{1}^{2}x_{2} + y_{1}^{2}y_{2}, \ \phi_{7} = x_{1}x_{2}^{2} + y_{1}y_{2}^{2}$\\
$-1$	& $\phi_{8} = x_{1}y_{2} - x_{2}y_{1}$ \\
        &$\phi_{9} = x_{1}^{3} - y_{1}^{3}, \ \phi_{10} = x_{2}^{3} - y_{2}^{3},\ \phi_{11} = x_{1}^{2}x_{2} - y_{1}^{2}y_{2}, \ \phi_{12} = x_{1}x_{2}^{2} - y_{1}y_{2}^{2}.$
\end{tabular}
\caption{generators of $\CC[V]^{H}$ that are eigenvectors of $\Ab(G)$}
\label{table:generatorsS3}
\end{table}

\begin{proposition}\label{proposition:commutator-invariantsS3}
The elements $\phi_{i}$ in Table~\ref{table:generatorsS3} are eigenvectors of the action of $\Ab(G)$ which generate the ring of invariants $\Pcal = \CC[V]^{H}\subset \CC[V]$.
\end{proposition}
\begin{proof}
We compute this generating set using Singular, \cite{Singular}. First we find the invariants of the action of~$[G,G]$. Then we split their linear span into smaller $\Ab(G)$-invariant subspaces on which the action of $\Ab(G)$ is easy to diagonalize.
\end{proof}

\begin{remark}
We have $\nu(\phi_{i}) = 0$ if $T(\phi_{i}) = \phi$ and $\nu(\phi_{i}) = 1$ if $T(\phi_{i}) = -\phi_{i}$. Since in this case $\nu(\phi_{i}) = D_{i}.C$ by~\eqref{eqn:intersections-vs-valuations}, then we have 
\begin{itemize}
\item $\overline{\phi}_{i} = \phi_{i}$ if $T(\phi) = \phi$,
\item $\overline{\phi}_{i} = \phi_{i}t$ if $T(\phi_{i}) = -\phi_{i}$.
\end{itemize}
\end{remark}

\begin{proposition}\label{proposition:cox-ringS3}
The image $\Theta(\Rcal(X))\subset \Pcal[t^{\pm 1}]$ of the Cox ring of the symplectic resolution $X\to V/G$ is generated by the elements $\phi_{1},\ldots, \phi_{7},\phi_{8}t,\ldots,\phi_{12}t$ and $t^{-2}$.
\end{proposition}

\begin{proof}
To use Theorem~\ref{theorem:general-theorem} we just have to show that condition~\ref{assumption:lifting-valuations} is satisfied. And this is the aim of the next section.
\end{proof}

\subsection{Lifting homogeneous elements}\label{section_lifting_S3}

Here we give an argument for condition~\eqref{assumption:lifting-valuations} in the case of considered representation of~$S_3$. In some sense this is the simplest possible case: there is only one conjugacy class of symplectic reflections, hence when chosing a lifting, we do not have to consider the interplay between different valuations~$\wt{\nu}_i$. However, the fact that there are three conjugate symplectic reflections leads to some complications which do not appear if the conjugacy classes have just two elements, as it happens for $\ZZ_{2}\wr S_{2}$, cf. section~\ref{section_D8_lifting}.

Let $\kappa:\CC[w_{1},\ldots, w_{12}]\to \Pcal$ be the surjective ring homomorphism which assigns $\phi_{i}$ to $w_{i}$. Define a linear map $\wt{T} \colon \bC^{12} \to \bC^{12}$ by setting
\begin{equation*}
\wt{T}(w_{1},\ldots,w_{12}) = (w_{1},\ldots,w_{7},-w_{8},\ldots,-w_{12})
\end{equation*}
and take the corresponding valuation $\wt{\nu}$, as in Definition~\ref{def-valuation-up}.

We start with noticing that valuations $\nu$ and $\wt{\nu}$ correspond to certain ideals generated by linear forms. Take the ideals of subspaces of $\bC^4$ fixed by symplectic reflections in~$G$:
\begin{equation*}
\begin{aligned}
& \cK = I(\ker(T-\id))= (x_{1} - y_{1}, x_{2} - y_{2}), \\
& \cK' = I(\ker(\varepsilon T - \id)) = (x_{1} - \epsilon y_{1},x_{2} - \epsilon y_{2}), \\
& \cK'' = I(\ker(\varepsilon^{-1}T -\id)) = (x_{1} - \epsilon^{-1}y_{1}, x_{2} - \epsilon^{-1}y_{2}). 
\end{aligned}
\end{equation*}
Also, set $\cJ = (w_{8},\ldots,w_{12})\triangleleft \CC[w_{1},\ldots,w_{12}]$.

\begin{lemma}\label{lemma:valuations-vs-ideals-S3}
The valuation $\nu$ is the valuation of the ideal $\cK$, that is $\nu(f) = d$ is equivalent to $f\in \cK^{d}$ and $f\not\in \cK^{d+1}$. The analogous statement is true for the valuation $\wt{\nu}$ and the ideal $\cJ$.
\end{lemma}

\begin{proof}
After a linear change of coordinates $T = \diag(-1,-1,1,1)$ we may write $\cK = (x_{1},y_{1})$. Then both statements $\nu(f) \ge d$ and $f\in \cK^{d}$ are equivalent to the fact that each monomial of $f$ is of degree at least $d$ with respect to $x_{1},y_{1}$.
Similar argument proves the statement concerning~$\wt{\nu}$.
\end{proof}

Recall that $\cP\simeq \cR(Y)$ is graded by $\Cl(Y) = G^{\vee} \simeq \Ab(G)^{\vee}$. We will consider elements homogeneous with respect to this grading.

\begin{lemma}\label{lemma:valuations-vs-ideals-S3-II}
If $f \in \cP$ is homogeneous with respect to $\Cl(Y)$-grading and $f \in \cK^d$ then $f\in \cK^{d}\cap (\cK')^{d}\cap (\cK'')^{d}$.
\end{lemma}

\begin{proof}
Symplectic reflections $T$, $\varepsilon T$ are conjugate, i.e. $\varepsilon T = gTg^{-1}$ for some $g\in G$. Thus $f\in \cK^{d} = I(\ker(T-\id))^{d}$ implies that $gf\in I(\ker (g(T-\id)g^{-1}))^{d} = (\cK')^{d}$. Since~$f$ is homogeneous, $G$ acts on $f$ by multiplication by some character $\chi \in G^{\vee}$. In particular, $gf = \chi(g) f$, where $\chi(g) \in \CC^{*}$. Hence $f\in (\cK')^{d}$, and in the same way $f\in (\cK'')^{d}$.
\end{proof}

\begin{remark}\label{lemma:lifting-valuationsS3}
We are going to prove condition~\eqref{assumption:lifting-valuations}, which in this case can be stated as follows: for each homogeneous $f\in \Pcal$ there exists $\wt{f}\in \CC[w_{1},\ldots,w_{12}]$ such that $\kappa(\wt{f}) = f$ and $\nu(f) = \wt{\nu}(\wt{f})$. 

Let~$\cI$ be the kernel of $\kappa$. By Lemma~\ref{lemma:valuations-vs-ideals-S3-II}, in terms of corresponding ideals this condition can be expressed as
\begin{equation}\label{equation-valuation-lifting}
\kappa^{-1}(\cK^{d}\cap (\cK')^{d}\cap (\cK'')^{d}) \subseteq \cJ^{d} + \cI.
\end{equation}
\end{remark}

The remaining part of this section is a proof of~\eqref{equation-valuation-lifting}. We start from a technical argument that in this case instead of considering intersection of powers of ideals we may look at powers of intersections.

\begin{lemma}\label{lemma:powers-and-intersections-S3}
For $d \in \bN$ we have $\cK^{d}\cap (\cK')^{d}\cap (\cK'')^{d} = (\cK\cap \cK'\cap \cK'')^{d}$.
\end{lemma}
\begin{proof}
After a linear change of coordinates we may assume that two of considered ideals are monomial: $\cK = (z_1, z_2)$, $\cK' = (z_3, z_4)$ and $\cK'' = (z_1+z_3, z_2+z_4)$ in $\bC[z_1, z_2, z_3, z_4]$.
Also, one checks directly (e.g. in Macaulay2, \cite{M2}) that $$\cK \cap \cK' \cap \cK'' = (z_1z_4-z_2z_3) + \cK \cdot \cK' \cdot \cK''.$$

Obviously the right-hand side is contained in the left-hand side. We prove the other inclusion. The proof goes by induction on $d$. For $d = 1$ the result is clear. We assume that $\cK^{d-1}\cap (\cK')^{d-1} \cap (\cK'')^{d-1} = (\cK \cap \cK' \cap \cK'')^{d-1}$ and prove the inclusion in
$$\cK^d\cap (\cK')^d \cap (\cK'')^d \subseteq (\cK^{d-1}\cap (\cK')^{d-1} \cap (\cK'')^{d-1})\cdot(\cK\cap \cK' \cap \cK'') = (\cK \cap \cK' \cap \cK'')^d.$$

Note that $(\cK'\cdot \cK'')^d = (\cK')^d \cap (\cK'')^d$, since these ideals are generated by linearly independent forms (i.e. after a linear change of coordinates they become monomial, generated in independent variables). Hence the left-hand side can be written as $\cK^d \cap (\cK'\cdot \cK'')^d$.
Consider an element~$\alpha$ of the left-hand side. We express it in terms of generators of $(\cK'\cdot \cK'')^d$ and then explore the information that it belongs to~$\cK^d$:
$$\alpha = \sum_{i\in I} C_i(z_1+z_3)^{a_i}(z_2+z_4)^{d-a_i}z_3^{b_i}z_4^{d-b_i}w_i,$$
where $w_i$ are monomials, $C_i$ are constants and $I$ is a set of indices. Assume that~$\alpha$ is homogeneous (with respect to standard degree) -- we can work in each degree separately. 

Now we show by induction that $\alpha$ can be written as a sum of element divisible by $z_1z_4-z_2z_3$ and an element of $(\cK \cap \cK' \cap \cK'')^d$.
We use the (descending) induction on the minimal pair $(m_i,n_i)$ such that $m_i$ is the degree of $w_i$ in $z_1$ and $z_2$ (i.e. sum of the degree in $z_1$ and in $z_2$) and $n_i$ is the degree of $w_i$ in $z_1$, ordered lexicographically.
We may assume that $\alpha$ contains components for which $(m_i,n_i) < (d,0)$ (i.e. $w_i$ is not divisible by any $z_1^cz_2^{d-c}$), since otherwise $\alpha \in (\cK \cap \cK' \cap \cK'')^d$.

Let $J \subseteq I$ be the set of indices with minimal $(m_i,n_i)$, equal to $(m,n)$. Look at
$$\alpha' = z_1^nz_2^{m-n}\sum_{i\in J} C_i (z_1+z_3)^{a_i}(z_2+z_4)^{d-a_i}z_3^{b_i}z_4^{d-b_i}z_3^{e_i}z_4^{f-e_i},$$
where $f$ is the same for all terms because of the homogeneity. It is enough to show that $\alpha' \in (z_1z_4-z_2z_3)$, the rest follows from the induction hypothesis for $\alpha - \alpha'$.

We do another (decreasing) induction on the minimal value of $a_i$ for $i \in J$. Let $a_0$ be this minimal value, assume first that $a_0 < d$.
Then for every component of the sum above with $a_i = a_0$ we modify $\alpha'$ to $\alpha'- (S_0 - S_1)$, where $S_0$ is the considered component (with exponents $b_0$ and $e_0$) and
$$S_1 = C_{i}(z_1+z_3)^{a_0+1}(z_2+z_4)^{d-a_0-1}z_3^{b_0+e_0-1}z_4^{d+f+1-b_0-e_0}.$$
We may assume that $b_0 + e_0 > 0$, because otherwise we have a monomial in $\alpha'$ which is not in $\cK^d$, but cannot cancel with any other: this is $c_0z_1^nz_2^{m-n}z_3^{a_0}z_4^{2d+f-a_0}$ (obtained by choosing always $z_3$ and $z_4$ from brackets).

Note that every time we subtract from $\alpha'$ a polynomial divisible by $z_1z_4-z_2z_3$:
$$S_0-S_1 = C_{i}(z_1+z_3)^{a_0}(z_2+z_4)^{d-a_0-1}z_3^{b_0 + e_0 - 1}z_4^{d+f-b_0-e_0}((z_2+z_4)z_3-(z_1+z_3)z_4).$$
Also, since $S_1$ is divisible by $(z_1+z_3)^{a_0+1}$, after all these modifications we have increased the minimal value of $a_i$. 

Finally, if $a_0 = d$ then the monomial obtained by choosing $z_3$ in every bracket in the corresponding sum element (remember that there is no $(z_2+z_4)$) in the component where $b_0+e_0$ is maximal is not in $\cK^d$ and it cannot delete with any other monomial.

Summing up, we can write $\alpha'$, hence also $\alpha$, in the form $(z_1z_4-z_2z_3)\beta + \gamma$, where $\gamma \in (\cK \cap \cK' \cap \cK'')^d$. Computing monomial valuations corresponding to $\cK, \cK', \cK''$ we get $\beta \in \cK^{d-1}$, $\beta \in (\cK')^{d-1}$, $\beta \in (\cK'')^{d-1}$, which ends the outer induction.

\end{proof}

\begin{lemma}\label{lemma:intersection-gens-S3}
The intersection $\cK\cap \cK'\cap \cK''$ is generated by the following polynomials
\begin{equation*}
{y}_{1} {x}_{2}-{x}_{1} {y}_{2}, \ {x}_{2}^{3}-{y}_{2}^{3}, \ {x}_{1} {x}_{2}^{2}-{y}_{1} {y}_{2}^{2}, \    {x}_{1}^{2} {x}_{2}-{y}_{1}^{2} {y}_{2}, \ {x}_{1}^{3}-{y}_{1}^{3}.
\end{equation*}

which are elements of~$\Pcal$. Moreover, we have
\begin{equation*}
(\cK\cap \cK'\cap \cK'')^{d}\cap \Pcal = (\cK\cap \cK'\cap \cK''\cap \Pcal)^{d}.
\end{equation*}
\end{lemma}

\begin{proof}
The first part is a direct computation, e.g. in Macaulay2. To prove the second part note that~$\Pcal$, as the ring of invariants of $[G,G]$, is generated by monomials of the form $x_{1}^{\alpha_{1}}y_{1}^{\beta_{1}}x_{2}^{\alpha_{2}}y_{2}^{\beta_{2}}$, where $\alpha_{1} + \alpha_{2} \equiv \beta_{1}+\beta_{2}\pmod{3}$. Only one inclusion needs to be proved: take $f\in (\cK\cap \cK'\cap \cK'')^{d}\cap \Pcal$, we show that it belongs to the right-hand side. Since $f\in (\cK\cap \cK'\cap \cK'')^{d}$, we may write 
\begin{equation*}
f = \sum_{i_{1}+\ldots + i_{5} = d}p_{1}^{i_{1}}\cdots p_{5}^{i_{5}}\alpha_{i_{1},\ldots, i_{5}},
\end{equation*}
where $p_{i}$ are generators of $\cK\cap \cK'\cap \cK''$ listed above and $\alpha_{i_{1},\ldots, i_{5}} \in \bC[x_1,y_1,x_2,y_2]$. And since $f \in \Pcal$, taking the average over~$[G,G]$ 
we get
\begin{equation*}
f = \frac{1}{|[G,G]|}\sum_{g \in [G,G]} gf = \sum_{i_{1}+\ldots + i_{5} = d}p_{1}^{i_{1}}\cdots p_{5}^{i_{5}}\overline{\alpha}_{i_{1},\ldots, i_{5}},
\end{equation*}
where $\overline{\alpha}_{i_{1},\ldots, i_{5}} = \frac{1}{|[G,G]|}\sum_{g\in G}g\alpha_{i_{1},\ldots,i_{5}}\in \Pcal$. And each element of this sum already belongs to $(\cK\cap \cK'\cap \cK''\cap \Pcal)^{d}$.
\end{proof}

Recall that $\cI = \ker \kappa$, take the pull-backs of considered ideals through $\kappa$:  $\wt{\cK} = \kappa^{-1}(\cK)$, $\wt{\cK}' = \kappa^{-1}(\cK')$, $\wt{\cK}'' = \kappa^{-1}(\cK'')$.

\begin{corollary}\label{corollary:preimage-of-intersection-S3} We have
\begin{equation*}
\kappa^{-1}((\cK\cap \cK'\cap \cK'')^{d}) = (\wt{\cK}\cap \wt{\cK}'\cap \wt{\cK}'')^{d} + \cI.
\end{equation*}
\end{corollary}
\begin{proof}
Recall that $\kappa$ is an epimorphism onto~$\Pcal$. In particular, if $A,A'$ are ideals of $\cP$ then $\kappa^{-1}(AA') = \kappa^{-1}(A)\kappa^{-1}(A') + \cI$. Using this fact together with Lemma~\ref{lemma:intersection-gens-S3} we obtain
\begin{multline*}
\kappa^{-1}((\cK\cap \cK'\cap \cK'')^{d}) = \kappa^{-1}((\cK\cap \cK'\cap \cK'')^{d}\cap \Pcal) =\\
= \kappa^{-1}((\cK\cap \cK'\cap \cK''\cap \Pcal)^{d}) = (\kappa^{-1}(\cK\cap \cK'\cap \cK''\cap \Pcal))^{d} + \cI = (\wt{\cK}\cap \wt{\cK}'\cap \wt{\cK}'')^{d} + \cI.
\end{multline*}
\end{proof}

The last observation can be verified by a direct computation in Macaulay2.

\begin{lemma}\label{lemma:intersection-of-preimages-S3}
$\wt{\cK}\cap \wt{\cK}'\cap \wt{\cK}'' = \cJ + \cI$.
\end{lemma}

Summing up, condition~\eqref{equation-valuation-lifting} follows by Lemma~\ref{lemma:powers-and-intersections-S3}, Corollary~\ref{corollary:preimage-of-intersection-S3} and Lemma~\ref{lemma:intersection-of-preimages-S3}. This finishes the proof of Proposition~\ref{proposition:cox-ringS3}.

%%%%%-----------------------------------------------------------------------

\section{{Wreath product $\ZZ_{2}\wr S_{2}$}}
\label{section:Z2wr}
\subsection{Representation}\label{section_D8_generators}

In this part we consider the symplectic action of the wreath product $G \simeq \ZZ_{2}\wr S_{2}$ on $V = \CC^{4}$ with coordinates $(x_1,x_2,x_3,x_4)$ and symplectic form $dx_{1}\wedge dx_{3} + dx_{2}\wedge dx_{4}$. The action is given by the embedding $\ZZ_{2} \subset \SL_{2}(\CC)$ and by permutation of factors $\CC^{4} = \CC^{2}\times \CC^{2}$. More precisely, $G$ is generated by matrices

\begin{equation*}
T_{0} =
\begin{pmatrix}
1 & 0 & 0 & 0\\
0 & -1 & 0 & 0\\
0 & 0 & 1 & 0\\
0 & 0 & 0 & -1
\end{pmatrix}
\qquad
T_{2} =
\begin{pmatrix}
0 & 1 & 0 & 0\\
1 & 0 & 0 & 0\\
0 & 0 & 0 & 1\\
0 & 0 & 1 & 0
\end{pmatrix}.
\end{equation*}

\begin{remark}
This is a subgroup of the 32-element group investigated in~\cite{cox_resolutions}. Our notation for symplectic reflections is compatible with the one used therein.
\end{remark}

Below we list basic properties of~$G$, the arguments are straightforward. Let $\nu_{0},\nu_{2}:\CC(V)^{[G,G]}\to \ZZ$ denote monomial valuations corresponding to symplectic reflections $T_{0},T_{2}$ respectively. 

\begin{proposition} \label{proposition:basic-properties-Z2wr} With notation as above:
\begin{itemize}
\item  $G$ is isomorphic to the dihedral group $D_{8}$ of order~$8$, the isomorphism is given by identifying elements $\pm T_{0},\pm T_{2}$ with reflections and $T_{0}T_{2}$ with rotation.
\item The commutator subgroup $H= [G,G] \simeq \bZ_2$ is generated by matrix $-\id$. In particular, $H$ does not contain any symplectic reflection.
\item $\Ab(G) = G/H = \langle T_{0}H,T_{2} H\rangle \cong \ZZ_{2}\times \ZZ_{2}$
\item Matrices $\pm T_{0}$ and $\pm T_{2}$ are all the symplectic reflections in~$G$. Their conjugacy classes are $\{\pm T_{0}\},\ \{\pm T_{2}\}.$ 
In particular, by the McKay correspondence the exceptional divisor of a symplectic resolution $\varphi:X\to Y = V/G$ has two irreducible components $E_{0},E_{2}$ corresponding to these conjugacy classes.
\item The intersection matrix $(E_{i}.C_{j})_{i,j}$ is a direct sum of two $A_{1}$-type Cartan matrices, i.e.
\begin{equation*}
(E_{i}.C_{j})_{i,j} = \begin{pmatrix}
-2 & 0\\
0 & -2
\end{pmatrix}
\end{equation*}
\item $\nu_{i}|_{\CC(V)^{H}} = 2\nu_{E_{i}}, \ i=0,2$ by Lemma~\ref{lemma:restricting-monomial-valuation}.
\end{itemize}
\end{proposition}

\begin{table}[h!]
\begin{tabular}{l r r}
generators &  & eigenvalues  \\
 & $T_{0}$ & $T_{2}$\\
$\phi_{13} = -x_{1}^{2} - x_{2}^{2} + x_{3}^{2} + x_{4}^{2}$ & $1$ & $1$\\
$\phi_{14} = \sqrt{-1}(x_{1}^{2} + x_{2}^{2} + x_{3}^{2} + x_{4}^{2})$ & $1$ & $1$\\
$\phi_{34} = 2(x_{1}x_{3} + x_{2}x_{4})$ & $1$ & $1$\\
$\phi_{01} = -2(x_{1}x_{4} + x_{2}x_{3})$ & $-1$ & $1$\\
$\phi_{03} = 2\sqrt{-1}(x_{1}x_{2} + x_{3}x_{4})$ & $-1$ & $1$\\
$\phi_{04} = 2(-x_{1}x_{2} + x_{3}x_{4})$ & $-1$ & $1$\\
$\phi_{12} = 2(x_{1}x_{3} - x_{2}x_{4})$ & $1$ & $-1$\\
$\phi_{23} = \sqrt{-1}(-x_{1}^{2} + x_{2}^{2} - x_{3}^{2} + x_{4}^{2})$ & $1$ & $-1$\\
$\phi_{24} = x_{1}^{2} - x_{2}^{2} - x_{3}^{2} + x_{4}^{2}$ & $1$ & $-1$\\
$\phi_{02} = 2\sqrt{-1}(x_{1}x_{4} - x_{2}x_{3})$ & $-1$ & $-1$\\

\end{tabular}
\caption{Generators of $\CC[V]^{H}$ which are eigenvectors of $\Ab(G)$ (cf.~\cite[3.13]{cox_resolutions}).}
\label{table:generatorsZ2wr}
\end{table}

\begin{proposition}
The elements $\phi_{ij}$ in Table~\ref{table:generatorsZ2wr} are eigenvectors of the action of $\Ab(G)$, which generate the ring of invariants $\CC[V]^{H}\subset \CC[V]$.
\end{proposition}
\begin{proof}
See~\cite[Lem.~3.12]{cox_resolutions}. These polynomials are eigenvectors of the action of the abelianization of the 32-element group investigated therein. Since~$G$ is its subgroup and commutator subgroups are equal, they are also eigevectors for $\Ab(G)$.
\end{proof}

\begin{remark}\label{remark_def_phi_ij}
Using~\eqref{eqn:intersections-vs-valuations} we obtain an explicit description of generators $\ovl{\phi}_{ij}$ of the Cox ring of symplectic resolutions of $V/G$, proposed in Theorem~\ref{theorem:general-theorem}:
\begin{equation}
\overline{\phi}_{ij} =\begin{cases} \phi_{ij}t_{0}t_{2} & \hbox{if }\{0,2\}= \{i,j\},\\ \phi_{ij}t_{0} &  \hbox{if } \{0,2\}\cap\{i,j\} = \{0\},\\ \phi_{ij}t_{2} & \hbox{if } \{0,2\}\cap \{i,j\} = \{2\}, \\ \phi_{ij}, & \hbox{if } \{0,2\}\cap \{i,j\} = \emptyset.
\end{cases}
\end{equation}
\end{remark}

To simplify the notation we set $\phi_{21} = \phi_{12}$.

\begin{proposition}\label{proposition:CoxRingZ2wr}
Elements $t_{0}^{-2}, t_{2}^{-2},\phi_{02}t_{0}t_{2}, \phi_{0i}t_{0}, \phi_{2j}t_{2}, \phi_{ij}$, for $i,j=1,3,4$, $i<j$ generate the image $\Theta(\Rcal(X))\subset \Pcal[t_{0}^{\pm 1},t_{2}^{\pm 1}]$ of the Cox ring of symplectic resolutions $X\to V/G$.
\end{proposition}

We will denote by~$\Rcal$ the ring generated by elements listed in Proposition~\ref{proposition:CoxRingZ2wr}.

In section~\ref{section_D8_lifting} we give a part of the proof of Proposition~\ref{proposition:CoxRingZ2wr}. A complete argument, based on different methods than these introduced here, will be presented in~\cite{valuationslift}. Next, in sections \ref{section_D8_ideal}-\ref{section_D8_central_fiber}, we investigate the geometry of certain GIT quotients of $\Spec \cR$: we prove their smoothness and provide an explicit description of the central fiber of a resolution.

\subsection{Lifting homogeneous elements}\label{section_D8_lifting}

Here we prove a weaker version of condition~\eqref{assumption:lifting-valuations} for $G = \ZZ_{2}\wr S_{2}$. In Lemma~\ref{lemma:lifting-valuations-separatedlyZ2wr} we show that it is satisfied for each valuation $\nu_{0}, \nu_2$ separately. The proof goes along the same lines as the one for $S_3$ in section~\ref{section_lifting_S3}. However, the argument is simpler since in $\ZZ_{2}\wr S_{2}$ each conjugacy class of symplectic reflections has only two elements. 

Let $\kappa\colon \CC[w_{ij}\colon 0 \leq i < j \leq 4]\to \Pcal$ be the surjective ring homomorphism which sends $w_{ij}$ to the eigenvectors $\phi_{ij}$ from Table~\ref{table:generatorsZ2wr}.

Define linear maps $\wt{T}_0, \wt{T}_2\colon \CC^{10}\to \CC^{10}$:
$$
\wt{T}_{0}(e_{ij}) = \begin{cases}e_{ij} & \hbox{if } 0\not\in \{i,j\}\\ -e_{ij} &  \hbox{if } 0\in \{i,j\}\end{cases}\qquad\qquad
\wt{T}_{2}(e_{ij}) = \begin{cases}e_{ij} & \hbox{if } 2\not\in \{i,j\}\\ -e_{ij} & \hbox{if } 2\in \{i,j\}\end{cases} 
$$
where $\{e_{ij} \colon 0 \leq i < j \leq 4\}$ is a basis of $\CC^{10}$ dual to the set of variables $w_{ij}$. By $\wt{\nu}_0, \wt{\nu}_2$ we denote corresponding monomial valuations.

Note that considered monomial valuations correspond to ideals generated by linear forms. For each conjugacy class of symplectic reflections in~$G$ and for each representative we write the ideal of the subspace of~$\CC^{4}$ fixed by this reflection:
\begin{equation*}
\begin{aligned}
& \cK_{0} := I(\ker(\id - T_{0}))= (x_{1}, x_{3}), \\
& \cK'_0 := I(\ker(\id + T_{0}))= (x_{2}, x_{4}), \\
& \cK_{2} := I(\ker(\id - T_{2}))= (x_{1} - x_{2}, x_{3} - x_{4}), \\
& \cK'_{2} := I(\ker(\id + T_{2})) = (x_{1} + x_{2},x_{3} + x_{4}).
\end{aligned}
\end{equation*}

Also, define ideals $\cJ_{0},\cJ_{2} \triangleleft \CC[w_{ij}\colon 0 \leq i < j \leq 4]$ by setting $\cJ_{0} = (w_{01},w_{02},w_{03},w_{04})$ and $\cJ_{2} = (w_{02},w_{12},w_{23},w_{24})$. The following properties can be proved in the same way as Lemmas~\ref{lemma:valuations-vs-ideals-S3} and~\ref{lemma:valuations-vs-ideals-S3-II}.

\begin{lemma}\label{lemma:valuations-vs-ideals-Z2wr}
The valuation $\nu_{i}$ is the valuation of ideal $\cK_{i}$, that is $\nu_{i}(f) = d$ is equivalent to $f\in \cK_{i}^{d}$ and $f\not\in \cK_{i}^{d+1}$ for $i=0,2$. The analogous statement is true for valuations $\wt{\nu}_{i}$ and ideals $\cJ_{i}$.
\end{lemma}
 
\begin{lemma}\label{lemma_representatives_intersection}
If $f \in \cP$ is homogeneous with respect to the $\Cl(Y)$-grading and $f \in \cK_i^d$, i.e. $\nu_i(f) \geq d$, then $f\in \cK_i^d \cap (\cK_i')^d$. 
\end{lemma}

Next we investigate the intersections $\cK_i^d \cap (\cK_i')^d$ and their preimages under~$\kappa$.

\begin{lemma}\label{lemma:powers-and-intersections-Z2wr}
For $d \in \bN$ and $i=0,2$ we have
$(\cK_{i}\cap \cK'_{i})^{d} = \cK_{i}^{d}\cap (\cK'_{i})^{d}$.
\end{lemma}
\begin{proof}
It is sufficient to notice that after a suitable change of coordinates $\cK_{i}$ and $\cK'_{i}$ are monomial ideals generated by independent variables.
\end{proof}

\begin{lemma}\label{lemma:intersection-gens-Z2wr}
We have $\cK_{0}\cap \cK'_{0}=(x_{1}x_{2}, x_{1}x_{4}, x_{2}x_{3},x_{3}x_{4})$ and $\cK_{2}\cap \cK'_{2} = ((x_{1}-x_{2})(x_{1} + x_{2}), (x_{3} - x_{4})(x_{3}+x_{4}), (x_{1} - x_{2})(x_{3} + x_{4}), (x_{3} - x_{4})(x_{1} + x_{2}))$. In particular, both ideals $\cK_{i}\cap \cK'_{i}$, $i=0,2$ are generated by elements from~$\Pcal$. Consequently
\begin{equation*}
(\cK_{i}\cap \cK'_{i})^{d}\cap \Pcal = (\cK_{i}\cap \cK'_{i}\cap \Pcal)^{d}.
\end{equation*}
\end{lemma}

\begin{proof}
The description of $\cK_{0}\cap \cK'_{0}$ is obvious and for $\cK_2\cap \cK'_2$ it follows easily because after a change of coordinates $\cK_2$ and $\cK'_2$ become monomial ideals generated in independent variables. The proof of the second part is based on the fact that $\cP$ is the ring of invariants of $[G,G]$, as in Lemma~\ref{lemma:intersection-gens-S3}.
\end{proof}

Define $\wt{\cK}_{i} = \kappa^{-1}(\cK_{i})$, $\wt{\cK}'_{i} = \kappa^{-1}(\cK'_{i})$ for $i=0,2$, and let $\cI = \ker \kappa$.

\begin{corollary}\label{corollary:preimage-of-intersection-Z2wr}
For $i=0,2$ Lemma~\ref{lemma:intersection-gens-Z2wr} gives
\begin{equation*}
\kappa^{-1}((\cK_{i}\cap \cK'_{i})^{d}) = (\wt{\cK}_{i}\cap \wt{\cK}'_{i})^{d} + \cI.
\end{equation*}
\end{corollary}

Finally, the following lemma is proved by a computation in Macaulay2,~\cite{M2}.
\begin{lemma}\label{lemma:intersection-of-preimages-Z2wr}
For $i=0,2$ we have $\cJ_{i} + \cI = \wt{\cK}_{i}\cap \wt{\cK}'_{i}$.
\end{lemma}

Summing up, we may lift homogeneous elements of~$\cP$ compatibly with one valuation.

\begin{lemma}\label{lemma:lifting-valuations-separatedlyZ2wr}
Let $\wt{\nu}_{i}$ be the monomial valuation of $\wt{T}_{i}, \ i=0,2$. Then, for $i = 0,2$ and every homogeneous $f\in \Pcal$ there exists $\wt{f}\in \CC[w_{ij}\colon 0\leq i<j \leq 4]$ such that $\kappa(\wt{f}) = f$ and $\nu_{i}(f) = \wt{\nu}_{i}(\wt{f})$.
\end{lemma}

\begin{proof}
Assume that $f \in \cK_i^d$. Then by Lemma~\ref{lemma_representatives_intersection} $f \in \cK_i^d\cap \cK_i'^d$, and by Corollary~\ref{corollary:preimage-of-intersection-Z2wr} and Lemma~\ref{lemma:intersection-of-preimages-Z2wr} we have
$$\kappa^{-1}(\cK_i^d \cap (\cK'_i)^d) = \kappa^{-1}((\cK_i\cap \cK'_i)^d) = (\widetilde{\cK}_i \cap \widetilde{\cK}'_i)^d + \cI = (\cJ_i + \cI)^d + \cI = \cJ_i^d + \cI.$$
\end{proof}

\subsection{The ideal of an embedding of {$\Spec \Rcal$} in an affine space}
\label{section_D8_ideal}

We are now going to describe the ideal of an embedding $\Spec \Rcal \hookrightarrow \CC^{12}$ and the relation with an analogous embedding constructed in case of the group of order~$32$ considered in~\cite{cox_resolutions}. 

To embed $\Spec \cR$ in $\bC^{12}$ we consider a surjective homomorphism
$$\Psi\colon \CC[w_{ij},u_{0},u_{2}\colon 0\leq i < j \leq 4] \to \Rcal$$ which sends $w_{ij}$to $\overline{\phi}_{ij}$, listed in Remark~\ref{remark_def_phi_ij}, and $u_k$ to $t_k^{-2}$ for $k = 0,2$.

If~$\Rcal'$ is the subring of $\Pcal[t_{0}^{\pm 1},\ldots, t_{4}^{\pm 1}]$ generated by $\phi_{ij}t_{i}t_{j}$, $0 \leq i < j \leq 4$ and $t_{k}^{-2}$, $k=0,\ldots,4$ then as in~\cite[Prop.~3.17]{cox_resolutions} we may consider a closed embedding $\Spec \Rcal'\to \CC^{15}$. It is given by a surjective ring homomorphism $$\Phi\colon \CC[w_{ij}, u_{k}\colon k=0,\ldots,4, 0\leq i < j \leq 4] \to \Rcal',$$ which sends $w_{ij}$ to $\phi_{ij}t_{i}t_{j}$ and $u_{k}$ to $t_{k}^{-2}$.

Note that $\Psi$ is the composition of $\Phi$ with the map
$\Pcal[t_{0}^{\pm 1},\ldots, t_{4}^{\pm 1}]\to \Pcal[t_{0}^{\pm 1},t_{2}^{\pm 1}]$ such that $t_{k}$ is mapped to 1 for $k = 1,3,4$ and to $t_k$ for $k=0,2$. Thus we obtain a commutative diagram of closed embeddings
\begin{equation}\label{diagram:caseD8-vs-caseBS}
\begin{CD}
\Spec \Rcal @>>> \CC^{12}\\
@VVV @VVV\\
\Spec \Rcal' @>>> \CC^{15}
\end{CD}
\end{equation}
where $\CC^{12}$ is an affine subspace of $\CC^{15}$ given by $u_{1} = u_{3} = u_{4} = 1$.

\begin{proposition}\label{proposition:ideal-of-SpecR}
The kernel of $\Psi$, i.e. the ideal of the embedding $\Spec \cR \hookrightarrow \bC^{12}$, is generated by the following polynomials:
$$\begin{array}{ll}
w_{14}w_{23} + w_{13}w_{24} - w_{12}w_{34} & w_{04}w_{23} - w_{03}w_{24} - w_{02}w_{34}\\ 
w_{04}w_{13} + w_{03}w_{14} - w_{01}w_{34} & w_{04}w_{12} - w_{02}w_{14} - w_{01}w_{24}\\
w_{03}w_{12} + w_{02}w_{13} - w_{01}w_{23} & \\
w_{02}w_{12}u_{2} - w_{03}w_{13} + w_{04}w_{14} & w_{01}w_{14} - w_{02}w_{24}u_{2} + w_{03}w_{34}\\
w_{01}w_{13} + w_{02}w_{23}u_{2} + w_{04}w_{34} & w_{01}w_{12} + w_{03}w_{23} + w_{04}w_{24}\\
w_{03}w_{04}u_{0} - w_{13}w_{14} + w_{23}w_{24}u_{2} & w_{02}w_{04}u_{0} + w_{12}w_{14} + w_{23}w_{34}\\
w_{01}w_{04}u_{0} + w_{12}w_{24}u_{2} + w_{13}w_{34} & w_{02}w_{03}u_{0} - w_{12}w_{13} - w_{24}w_{34}\\
w_{01}w_{03}u_{0} + w_{12}w_{23}u_{2} + w_{14}w_{34} & w_{01}w_{02}u_{0} + w_{13}w_{23} - w_{14}w_{24}\\
w^2_{02}u_{0} + w^{2}_{12} + w^{2}_{23} + w^{2}_{24} & w^{2}_{03}u_{0} + w^{2}_{13} + w^{2}_{23}u_{2} + w^{2}_{34} \\
w^{2}_{01} + w^{2}_{02}u_{2} + w^{2}_{03} + w^{2}_{04} & w^{2}_{04}u_{0} + w^{2}_{14} + w^{2}_{24}u_{2} + w^{2}_{34}\\
w^{2}_{01}u_{0} + w^{2}_{12}u_{2} + w^{2}_{13} + w^{2}_{14}. & \\
\end{array}$$

Note that according to~\cite[Prop.~3.17]{cox_resolutions} these polynomials are exactly the generators of $\ker \Phi$ after substituting $u_{i}\mapsto 1$ for $i=1,3,4$. In particular, $\Spec \Rcal = \CC^{12}\cap \Spec \Rcal'$.
\end{proposition}

\begin{proof}
This can be computed e.g. in Macaulay2,~\cite{M2}. Another way of obtaining the generators of the ideal of $\Spec \cR$ is to observe that the ideal of the closed embedding $\Spec \cR\hookrightarrow \Spec \cR'$ is equal to $(t_{1}^{-2} - 1, t_{3}^{-2}-1, t_{4}^{-2}-1)$ 
and use the generators of the ideal of $\Spec \cR'$ listed in~\cite[Prop.~3.17]{cox_resolutions}.
\end{proof}

\subsection{GIT quotients of $\Spec \cR$ -- linearization and stability}
\label{section_D8_stability}

In the following sections we are going to show that any crepant resolution $X\to V/G$ is a GIT quotient of $\Spec \Rcal$ by an action of a two-dimensional torus $\TT_{\Cl(X)}$, and to describe the central fiber of such a resolution. We start from investigating appropriate linearizations of this action.

Consider the two-dimensional Picard torus $\TT = \TT_{\Cl(X)}$ of $X$, which we identify with a subtorus of $\TT'\cong (\CC^{*})^{5}$ given by $t_{1} = t_{3} = t_{4}=1$. We have a natural action of~$\TT$ on $\Spec \Rcal$ induced by the inclusion $\Rcal \subset \Theta(\Rcal(X))\subset \cP[t_{0}^{\pm 1}, t_{2}^{\pm 1}]$. 

More explicitly, $\TT\cong \{(t_{0},t_{2}) \colon t_{0},t_{2}\in \CC^{*}\}$ acts trivially on $w_{13}, w_{14}, w_{34}$, by character $t_{0}$ on $w_{01}, w_{03}, w_{04}$, by character $t_{2}$ on $w_{12}, w_{23}, w_{24}$, by character $t_{0}t_{2}$ on $w_{02}$, by character $t_{0}^{-2}$ on $u_{0}$ and by character $t_{2}^{-2}$ on $u_{2}$.

\begin{notation}
By $\chi$ we denote the character $\chi^{(2,1)}$ of~$\bT$.
\end{notation}

Our goal is to show the results analogous to \cite[Prop.~4.12, 4.14]{cox_resolutions} for the $\TT$-action on $\Spec \Rcal$ linearized by~$\chi$. In the present section we investigate the (semi)stability for this linearization. The subsequent one explains the smoothness of the geometric quotient of the stable locus. Note that by symmetry the same will be true then for~$\chi^{(1,2)}$.

We describe explicitly the $\chi$-semistable locus of $\Spec \Rcal$, the proof is straightforward.

\begin{lemma}\label{lemma:semistability-for-D8}
The set of $\chi$-semistable points of $\CC^{12}$ is equal to
\begin{equation*}
(\CC^{12})_{\chi}^{ss} = \bigcup_{i=1,3,4}\{w_{02}w_{0i}\neq 0\}\cup \bigcup_{i,j=1,3,4} \{w_{0i}w_{2j}\neq 0\},
\end{equation*}
where we denote $w_{21} := w_{12}$.
\end{lemma}

The next observation may be verified directly, using the above characterization of $\chi$-semistable points.

\begin{proposition}\label{proposition:stability-for-D8}
Every $\chi$-semistable point of $\Spec \cR$ is $\chi$-stable and has trivial isotropy group.
\end{proposition}

\subsection{The smoothness of GIT quotients}\label{subsection_quotient_smoothness}
\label{section_D8_smoothness} 

We will now show that a crepant resolution~$X$ can be reconstructed as a GIT quotient of $\Spec \Rcal$. To prove the smoothness of the quotient we follow the reasoning from the proof of an analogous result~\cite[Prop.~4.14]{cox_resolutions}. However, our case is simpler, it requires less computations. The main step is showing the smoothness of $\chi$-stable locus of $\Spec \Rcal$. Then the smoothness of its geometric quotient follows from the facts that~$\bT$ acts freely on this set.

\begin{proposition}\label{proposition:smoothness-of-stable-pts-D8}
$\Spec \Rcal$ is smooth in every $\chi$-stable point.
\end{proposition}
\begin{proof}

We use the Jacobian matrix~$M$ of the generating set of the ideal $\ker \Psi$ of $\Spec \cR \subset \bC^{12}$ listed in Proposition~\ref{proposition:ideal-of-SpecR}. Its rows correspond to partial derivatives with respect to variables in the following order: 
\begin{equation*}
w_{01},w_{02},w_{03},w_{04},w_{12},w_{13},w_{14},w_{23},w_{24},w_{34},u_{0},u_{2}.
\end{equation*}

As in the proof of~\cite[Prop.~4.14]{cox_resolutions} we look for a $6\times 6$ minors of~$M$ which are monomials. Such minors give a vanishing of coordinates in singular points. We use the description of the set of stable points from Lemma~\ref{lemma:semistability-for-D8}) to conclude that points with certain sets of vanishing coordinates cannot be stable.

We use the following notation: $\det(i_{1},\ldots, i_{6}|j_{1},\ldots,j_{6})$ is the $6\times 6$ minor of Jacobian matrix $M$ of rows $i_{1},\ldots, i_{6}\in \{1,\ldots, 12\}$ and columns $j_{1},\ldots,j_{6}\in \{1,\ldots, 20\}$. By $\Mon(x_{i_{1}},\ldots,x_{i_{k}})$ we denote the set of monomials in variables $x_{i_{1}},\ldots, x_{i_{k}}$.

We start by observing that $\det(4,6,7,9,10,11 | 2,5,6,9,16,18)\in \Mon (w_{02}, w_{04})$ and also $\det(5,6,7,10,11,12 | 2,4,5,9,16,18)\in \Mon (w_{01},w_{02})$. Now if $w_{02}\neq 0$ then each singular point of $\Spec \Rcal$ satisfies $w_{01} = w_{04} = 0$, and substituting this zeroes to~$M$ we get $\det (6,7,8,10,11,12 | 2,3,5,9,16,18)\in \Mon (w_{02},w_{03})$. Given that $w_{02}\neq 0$ we get also $w_{03} = 0$ at each singular point. Together with $w_{01} = w_{04} = 0$ this implies that singular points satisfying $w_{02}\neq 0$ are not $\chi$-stable.

It remains to consider singular points satisfying $w_{02} = 0$. Substituting this to~$M$ we get $\det(1,6,7,8,9,11 | 4,5,7,8,18,20)\in \Mon(w_{01})$, hence $w_{01} = 0$. Substituting to~$M$ again, we have $\det (4,5,7,9,10,11 |4,6,8,9,18,19)\in \Mon(w_{04})$ and $\det(3,5,7,9,10,11| 2,3,5,7,17,18)\in \Mon(w_{03})$. Hence $w_{03} = w_{04} = 0$. Together with $w_{01}=0$ this, as before, implies that singular points satisfying $w_{02} = 0$ are not $\chi$-stable.
\end{proof}

Following~\cite[Prop.~3.16]{cox_resolutions}, we relate GIT quotients of $\Spec \cR$ to the singular space~$V/G$. 

\begin{lemma}\label{lemma:invariants-of-torus-D8}
The homomorphism $\Rcal\to \CC[V]^{[G,G]}$ induces an isomorphism $\Rcal^{\TT} \cong \CC[V]^{G}$.
\end{lemma}
\begin{proof}
Note that the group $\Ab(G)$ can be presented as a subgroup of $\TT$ by sending the class of $T_{0}$ to $(-1,1)$ and $T_{2}$ to $(1,-1)$. This makes the homomorphism $\Rcal\to \CC[V]^{[G,G]}$ an $\Ab(G)$-equivariant homomorphism, hence the image of $\Rcal^{\TT}$ is contained in $\CC[V]^{G}$. It is injective on $\Rcal^{\TT}$, since $\Rcal^{\TT}$ has trivial intersection with its kernel $(t_{0}-1,t_{2}-1)\cap \cR$. To prove it is onto, note that a $[G,G]$-invariant element  $\prod \phi_{ij}^{a_{ij}}$ is in $\CC[V]^{G}$ if and only if $a_{01}+a_{02}+a_{03} + a_{04} = 2s_{0}, \ a_{02} + a_{12} + a_{23} + a_{24} = 2s_{2}$ for some nonnegative integers $s_{0},s_{2}$. Then $\prod \phi_{ij}^{a_{ij}} = (t_{0}^{-2})^{s_{0}}(t_{2}^{-2})^{s_{2}}\prod \overline{\phi}_{ij}^{a_{ij}}$.
\end{proof}

\begin{theorem}\label{theorem:smoothness-D8}
The GIT quotient of $\Spec \Rcal$ associated with the linearization $\chi$ is a crepant resoloution of singularities of $V/G$. There are exactly two such resolutions, the other one is the GIT quotient of $\Spec \Rcal$ for the linearization $\chi^{(1,2)}$.

\end{theorem}

\begin{proof}
Lemma~\ref{lemma:invariants-of-torus-D8} gives a birational morphism $X\to V/G$. This morphism is projective by~\cite[Prop.~14.1.12]{CLS} applied to $\Spec \Rcal$ and its ambient affine space. So it suffices to prove the smoothness of the GIT quotient. By Lemma~\ref{proposition:stability-for-D8}, the action of~$\TT$ on the set of $\chi$-stable points of $\Spec \Rcal$ is free. By Corollary~\ref{proposition:smoothness-of-stable-pts-D8} we know that the set of $\chi$-stable points of $\Spec \Rcal$ is nonsingular. Now we may deduce nonsingularity of the quotient since the geometric quotient of a nonsingular variety by a free torus action is nonsingular.

To show that this resolution is crepant it suffices to note that it gives crepant, i.e. minimal, resolution of $A_{1}$-singularities in codimension~2. This is true since generic fibers of the resolution restricted to exceptional divisor are irreducible.

To prove the second part we note that by~\cite[Thm.~1.2]{WierzbaWisniewski} any two such resolutions differ by finite sequence of Mukai flops with respect to projective planes contained in the central fiber. Since we prove in section~\ref{section_D8_central_fiber} that there is only one such plane, there are exactly two symplectic resolutions. It remains to observe that changing linearization from $\chi^{(2,1)}$ to $\chi^{(1,2)}$ yields a Mukai flop of GIT quotients.
\end{proof}

\subsection{The central fiber}
\label{section_D8_central_fiber}

Let $\varphi\colon X\to V/G$ be one of the resolutions obtained as (nonsingular) GIT quotients of $\Spec \cR$ by Theorem~\ref{theorem:smoothness-D8}. We proceed to describe the central fiber $S = \varphi^{-1}([0])$, i.e. the fiber over the image of $0$ via the quotient map $V\to V/G$. It is the unique $2$-dimensional fiber of $\varphi$. By the McKay correspondence $S$ has two components. Indeed, in~$G\cong D_{8}$ we have two more conjugacy classes apart from two formed by symplectic reflections (corresponding to the components of exceptional divisor), and the class of the identity (corresponding to the whole variety $X$).  

\begin{proposition}\label{proposition:central-fiber-D8}
The central fiber of~$\varphi$ decomposes as $S = F\cup P$ where~$F$ is isomorphic to the Hirzebruch surface~$\FF_{4}$ and $P\cong \PP^{2}$. The intersection $F\cap P$ is a smooth quadric on~$P$. 
\end{proposition}

We will verify this statement directly by describing the preimage of~$[0]$ in $\Spec \Rcal$ and then investigating its quotient by~$\bT$. Let~$W$ be the preimage of $[0]$ under the map $\Spec \Rcal \to V/G$ induced by the isomorphism $\CC[V]^{G} \cong \Rcal^{\TT} \subseteq \cR$. If $X\to V/G$ is one of the resolutions constructed in the previous section, we have the commutative diagram:
\begin{equation*}
\xymatrix@C-2pc@R-0.5pc{
W\ar[dd]\ar@{-->}[rrrd] & \subset & \Spec \Rcal\ar[dd]\ar@{-->}[rrrd] & & & \\
& & & S\ar[llld]& \subset & X \ar[llld]^{\varphi}\\
[0] & \in & V/G & &
}
\end{equation*}

\begin{lemma}\label{lemma:components-over-central-fiber}
$W \subseteq \Spec \cR \subseteq \bC^{12}$ decomposes into irreducible varieties as $W= W_u\cup W_{02}\cup W_{0}\cup W_{2}$, where $\dim W_u = 2$ and $dim W_{02} = \dim W_{0} = \dim W_{2} = 4$. More precisely:
\begin{itemize}
\item $W_u$ is a two-dimensional affine subspace of $\CC^{12}$ given by $\{w_{ij} = 0 \colon 0\leq i < j \leq 4\}$.
\item $W_{02}$ is a four-dimensional subvariety of $\CC^{12}$ with ideal generated by
$$
\begin{array}{ll}
w_{01}w_{23} - w_{03}w_{12} & w_{01}^{2} + w_{03}^{2} + w_{04}^{2}\\
w_{03}w_{24} - w_{04}w_{23} & w_{12}^{2} + w_{23}^{2} + w_{24}^{2}\\
w_{04}w_{12} - w_{01}w_{24} & w_{01}w_{12} + w_{03}w_{23} + w_{04}w_{24}\\
w_{13},\ w_{14},\ w_{34},\ u_{0},\ u_{2}. &
\end{array}
$$
\item $W_{0}$ is a four-dimensional subvariety of $\CC^{12}$ with ideal generated by
\begin{equation*}
w_{01}, \ w_{03}, \ w_{04},\ w_{13},\ w_{14},\ w_{34},\ u_{2},\ w_{02}^{2}u_{0} + w_{12}^{2} + w_{23}^{2} + w_{24}^{2}.
\end{equation*}
\item $W_{2}$ is a four-dimensional subvariety of $\CC^{12}$ with ideal generated by
\begin{equation*}
w_{12}, \ w_{23}, \ w_{24},\ w_{13},\ w_{14},\ w_{34},\ u_{0},\ w_{02}^{2}u_{2} + w_{01}^{2} + w_{03}^{2} + w_{04}^{2}.
\end{equation*}
\end{itemize}
\end{lemma}

\begin{proof}
The irreducible components of $W$ are computed in Macaulay2,~\cite{M2}, by finding the primary decomposition of the image in $\Rcal$ of the ideal of point $[0]\in V/G$. This is the ideal of $\Rcal$ generated by invariants of the $\TT$-action on~$\Rcal$.
\end{proof}

In what follows we take the linearization of $\TT$ corresponding to $\chi = \chi^{(2,1)}$. Results and proofs in the case of $\chi^{(1,2)}$ are essentially the same. We start from describing the set of $\chi$-stable points of~$W$.

\begin{lemma}\label{lemma:stability-over-central-fiber}
The set of $\chi$-semistable points on each component of $W$ coincides with the set of $\chi$-stable points on this component and is equal to its intersection with $(\CC^{12})^{ss}$. In particular $W^{ss} = W^{s} = W_{2}^{s} \cup W_{02}^{s}$.
\end{lemma}
\begin{proof}
This follows from the description of $\chi$-stable points in Lemma~\ref{lemma:semistability-for-D8}.
\end{proof}

Now we look at the quotient of the set of stable points in $W_2$, which gives the component isomorphic to~$\bP^2$.

\begin{proposition}\label{proposition:PP2-component}
The quotient of~$W_{2}^{s}$ by the action of~$\TT$ is isomorphic to~$\PP^{2}$. The quotient of the intersection $W_{02}^{s}\cap W_{2}^{s}$ is mapped by this isomorphism to a smooth quadric in~$\PP^{2}$.
\end{proposition}
\begin{proof}
By Lemma~\ref{lemma:components-over-central-fiber} we have
\begin{equation*}
W_{2} \cong\Spec \CC[w_{02},w_{01},w_{03},w_{04},u_{2}]/(w_{02}^{2}u_{2} + w_{01}^{2} + w_{03}^{2} + w_{04}^{2}).
\end{equation*}
Lemma~\ref{lemma:stability-over-central-fiber} gives an open cover of $W_{2}^{s}$ consisting of three open subsets defined by conditions $w_{02}w_{01}\neq 0$, $w_{02}w_{03}\neq 0$, $w_{02}w_{04}\neq 0$. In particular, $W_{2}^{s}$ is contained in the affine open subset given by $w_{02}\neq 0$, which is clearly isomorphic to
$\Spec \CC[w_{02},w_{01},w_{03},w_{04}]_{w_{02}}$.

Consider the ring homomorphism
$\CC[z_{1},z_{3},z_{4}]\to \Spec \CC[w_{02},w_{01},w_{03},w_{04}]_{w_{02}}$ defined by $z_{i}\mapsto w_{0i}$ for $i=1,3,4$.

It becomes the graded ring homomorphism if we take the standard $\ZZ$-grading on $\CC[z_{1},z_{3}, z_{4}]$ associated with the $\CC^{*}$ action by homotheties, the $\ZZ^{2}$-grading induced by $\TT$-action on the second ring and the monomorphism $\ZZ\to \ZZ^{2}$ given by $n \mapsto (n,0)$. Thus the corresponding morphism of varieties is equivariant with respect to considered torus actions.

Moreover, one may verify that it induces the isomorphism of quotients by looking at open covers $\{z_{i}\neq 0\}$ and  $\{w_{0i}\neq 0\}$ for $i=1,3,4$ of sets of stable points. Hence indeed we obtain~$\PP^{2}$ as a quotient.

To prove the second statement, observe that the intersection $W_{02}\cap W_{2}$ is defined in the open subset $\{w_{02}\neq 0\}$ of $W_{2}$ by equation $w_{01}^{2} + w_{02}^{2} + w_{03}^{2} = 0$. This corresponds to the subvariety of $\PP^{2}$ given by homogeneous equation $z_{1}^{2} + z_{3}^{2} + z_{4}^{2} = 0$ which is a smooth quadric.
\end{proof}

Finally, we recover the second component, which is present in the central fiber for both constructed resolutions.

\begin{proposition}\label{proposition:FF4-component}
The quotient of $W_{02}^{s}$ by the action of $\TT$ is isomorphic to the Hirzebruch surface $\FF_{4}$.
\end{proposition}
\begin{proof}
We may consider $W_{02}$ as an affine subvariety of $\CC^{7}$ with coordinates $w_{02},w_{01},w_{03},w_{04},w_{12},w_{23},w_{24}$. After substituting $w'_{03} = -w_{03} - iw_{04}$, $w'_{04} = w_{03} - iw_{04}$, $w'_{23} = -w_{23} - iw_{24}$, $w'_{24} = w_{23} - iw_{24}$ 
we see that $W_{02}$ is defined by a binomial ideal generated by 
$$
\begin{array}{llll}
w_{01}w_{23} - w_{03}w_{12} & w_{03}w_{24} - w_{04}w_{23} & w_{04}w_{12} - w_{01}w_{24} & w_{01}w_{12} - w_{03}w_{24}\\
w_{01}^{2} - w_{03}w_{04} & w_{12}^{2} - w_{23}w_{24} & & \\
\end{array}
$$
where, by abuse of notation, new coordinates are again called $w_{ij}$. Thus $W_{02}$ is an affine toric subvariety of $\CC^{7}$.

It corresponds to the cone $\sigma \subset N_{\bR} \simeq \RR^{4}$ (where the one-parameter subgroup lattice~$N$ is the standard lattice $\bZ^4 \subseteq \bR^4$) with rays
\begin{equation*}
(0,1,1,-1),\ (0,1,2,-1), \ (0,0,0,1), \ (0,1,0,1), \ (1,-1,-1,1).
\end{equation*}
One may check it by computing the Hilbert basis of the dual cone and the corresponding ideal of relations, e.g. in Macaulay2,~\cite{M2}. Denote the associated big torus by~$\TT_N$.

By Lemma~\ref{lemma:stability-over-central-fiber} the complement of the set of $\chi$-stable points on~$W_{02}$ is a sum of two irreducible toric, hence also $\bT_N$-invariant, subvarieties 
$$Y_{1} = V(w_{01},w_{03},w_{04},w_{12}^{2} - w_{23}w_{24}), \qquad Y_{2} = V(w_{02},w_{12},w_{23},w_{24},w_{01}^{2} - w_{03}w_{04})$$ of $\CC^{7}$. Note that $\codim_{W_{02}} Y_{1} = 1$ and $\codim_{W_{02}} Y_{2} = 2$.

Since $Y_{1},Y_{2}$ are irreducible and $\TT_N$-invariant, they are closures of some orbits $O_{1},O_{2}$ of $\TT_N$-action on $W_{02}$. They correspond to faces $\sigma_{1},\sigma_{2}$ of $\sigma$ of dimension~$1$ and~$2$ respectively.

The description of the ideal of the intersection of orbit closure with open affine cover in terms of the fan of a toric variety, see~\cite[3.2.7]{CLS}, gives us in particular the conditions
\begin{align*}
& (1,1,0,0),(1,0,1,0),(1,2,-1,0) \in (\sigma_{1}^{\vee} \setminus \sigma_{1}^{\perp}),\\
& (1,0,0,0), (0,0,1,1), (0,-1,2,1), (0,1,0,1)\in (\sigma_{2}^{\vee} \setminus \sigma_{2}^{\perp}).
\end{align*}

This suffices to deduce that~$\sigma_{1}$ is spanned by $(0,1,1,-1)$ and~$\sigma_{2}$ by $(0,0,0,1)$ and $(1,-1,-1,1)$.

Now we may describe the fan of the quotient of $W_{02}^{s}$ by the $\TT$-action as follows (cf.~\cite[Thm.~5.1]{Hamm}). We take the fan obtained by removing all cones containing either $\sigma_1$ or $\sigma_2$ from the fan of $\sigma$ and compute its image under the homomorphism $N_{\bR} \to \bR^2$ corresponding to the quotient of~$\bT_N$ by~$\bT$.

On the level of monomial lattices the embedding $\TT\cong (\CC^{*})^{2}\subset \TT_{W_{02}}$ can be given by characters corresponding to the columns of
\begin{equation*}
\begin{pmatrix}
1 & 0 & 0 & 0\\
1 & -1 & -1 & 2
\end{pmatrix}.
\end{equation*} 
This can be understood by looking at the action of~$\bT$ in the original coordinates on~$\bC^7$ and rewriting it in terms of a basis of the monomial lattice~$M$ of~$\bT_N$.

Thus on the level of one-parameter subgroup lattices we take the homomorphism given by a basis of the kernel of the matrix above:
\begin{equation*}
\begin{pmatrix}
0 & -1 & 1 & 0\\
0 & 2 & 0 & 1
\end{pmatrix}.
\end{equation*}

One can compute now that the images of faces of $\sigma$ not containing $\sigma_1$ nor $\sigma_2$ form a complete fan in~$\ZZ^{2}$ with rays $(1,1),(1,0),(-1,3),(-1,0)$, which is the fan of $\FF_{4}$.
\end{proof}

Proposition~\ref{proposition:central-fiber-D8} is thus a direct consequence of Propositions~\ref{proposition:PP2-component} and~\ref{proposition:FF4-component}.

%%%%%-----------------------------------------------------------------------

\section{{Binary tetrahedral group $G_4$}}
\label{section:G4}

Consider the symplectic representation $G = G_{4}\subset \Sp(\CC^{4},\omega)$ of the binary tetrahedral group which is generated by the following two matrices:
\begin{equation*}
\iota =
\begin{pmatrix}
i & 0 & 0 & 0\\
0 & -i & 0 & 0\\
0 & 0 & i & 0\\
0 & 0 & 0 & -i\\
\end{pmatrix}, \qquad
\tau = -\frac{1}{2}
\begin{pmatrix}
(1+i)\epsilon & (-1+i)\epsilon & 0 & 0\\
(1+i)\epsilon & (1-i)\epsilon & 0 & 0\\
0 & 0 & (1+i)\epsilon^{2} & (-1+i)\epsilon^{2} \\
0 & 0 & (1+i)\epsilon^{2} & (1-i)\epsilon^{2} \\
\end{pmatrix},
\end{equation*}
where $\epsilon$ is a $3$rd root of unity and $\omega = dx_{1}\wedge dy_{1} + dx_{2}\wedge dy_{2}$ (see~\cite{LehnSorger}).
We sum up the basic properties of this representation in the next proposition.

\begin{proposition}\leavevmode
\begin{itemize}
\item The commutator subgroup $H = [G,G]$ is isomorphic to quaternion group $Q_{8}$ and is generated by matrices $\iota, \xi$, where
\begin{equation*}
\xi = \begin{pmatrix}
0 & i & 0 & 0\\
i & 0 & 0 & 0\\
0 & 0 & 0 & i\\
0 & 0 & i & 0
\end{pmatrix}
\end{equation*}
In particular it does not contain any symplectic reflection.

\item The quotient group $\Ab(G) = G/H$ is isomorphic to the group $\ZZ_{3}$ and is generated by the image of $\tau$.

\item There are two four-elements conjugacy classes of symplectic reflections, represented by $T_{1} = \tau$ and $T_{2} = \tau^{2}$. Hence by the McKay correspondence there are two components $E_{1},E_{2}$ of the exceptional divisor corresponding to conjugacy classes of symplectic reflections $T_{1},T_{2}$ respectively.

\item The singular locus of the quotient $V/G$ is irreducible and consists of one transversal $A_2$ singularity since the symplectic reflection $\tau$ is of order $3$ and there are no reflections in the commutator subgroup (see Remark~\ref{remark:An-type-singularities}). In particular the intersection matrix $(E_{i}.C_{j})_{i,j}$ is a Cartan matrix of type $A_2$:
\begin{equation*}
\begin{pmatrix}
-2 & 1\\
1 & -2
\end{pmatrix}
\end{equation*}
\item Let $\nu_{i}$ be the monomial valuation $\CC(V)^{[G,G]}\to \ZZ$ associated with symplectic reflection $T_{i}$. Then $\nu_{i}|_{\CC(V)^{G}} = 3\nu_{E_{i}}, \ i=1,2$, by Lemma~\ref{lemma:restricting-monomial-valuation}.
\end{itemize}
\end{proposition}

\begin{table}[h!]
\begin{tabular}{l}
eigenvalue $1$: \\
\\
$\phi_1 = {y}_{1} {x}_{2}-{x}_{1} {y}_{2}$ \\
$\phi_2 =  {x}_{2}^{5} {y}_{2}-{x}_{2} {y}_{2}^{5}$ \\
$\phi_3 = {x}_{1}^{5} {y}_{1}-{x}_{1} {y}_{1}^{5}$ \\

$\phi_{4} = {x}_{1}^{4}+(-4 a^{2}+2) {x}_{1}^{2} {y}_{1}^{2}+{y}_{1}^{4}$\\
$\phi_{5} = {x}_{2}^{4}+(4 a^{2}-2) {x}_{2}^{2} {y}_{2}^{2}+{y}_{2}^{4}$\\
$\phi_{6}= {x}_{1} {x}_{2}^{3}+(-2 a^{2}+1) {y}_{1} {x}_{2}^{2} {y}_{2}+(-2 a^{2}+1)
       {x}_{1} {x}_{2} {y}_{2}^{2}+{y}_{1} {y}_{2}^{3}$\\
$\phi_{7} =  {x}_{1}^{3} {x}_{2}+(2 a^{2}-1) {x}_{1} {y}_{1}^{2} {x}_{2}+(2 a^{2}-1)
       {x}_{1}^{2} {y}_{1} {y}_{2}+{y}_{1}^{3} {y}_{2}$\\
$\phi_{8} = (2 a^{2}-2) {x}_{1}^{2} {y}_{1} {x}_{2}^{3}+(2 a^{2}-2) {x}_{1}^{3}
       {x}_{2}^{2} {y}_{2}+(-2 a^{2}+2) {y}_{1}^{3} {x}_{2} {y}_{2}^{2}+(-2
       a^{2}+2) {x}_{1} {y}_{1}^{2} {y}_{2}^{3}$\\
\\
eigenvalue $\epsilon$:\\
\\
$\phi_9=-a^{2} {x}_{1}^{2} {x}_{2}^{2}+(-(1/3) a^{2}+2/3) {y}_{1}^{2} {x}_{2}^{2}+(-(4/3) a^{2}+8/3) {x}_{1} {y}_{1} {x}_{2} {y}_{2}+$\\
$+(-(1/3)a^{2}+2/3) {x}_{1}^{2} {y}_{2}^{2}-a^{2} {y}_{1}^{2} {y}_{2}^{2}$
\\
$\phi_{10} = {x}_{2}^{4}+(-4 a^{2}+2) {x}_{2}^{2} {y}_{2}^{2}+{y}_{2}^{4}$\\
$\phi_{11}={x}_{1}^{3} {x}_{2}+(-2 a^{2}+1) {x}_{1} {y}_{1}^{2} {x}_{2}+(-2 a^{2}+1)
       {x}_{1}^{2} {y}_{1} {y}_{2}+{y}_{1}^{3} {y}_{2}$\\
$\phi_{12}=-10 {x}_{1}^{4} {y}_{1} {x}_{2}+2 {y}_{1}^{5} {x}_{2}-2 {x}_{1}^{5}
       {y}_{2}+10 {x}_{1} {y}_{1}^{4} {y}_{2}$\\
       
$\phi_{13}= 2 {x}_{1} {y}_{1} {x}_{2}^{4}+4 {x}_{1}^{2} {x}_{2}^{3} {y}_{2}-4
       {y}_{1}^{2} {x}_{2} {y}_{2}^{3}-2 {x}_{1} {y}_{1} {y}_{2}^{4}$\\
\\
eigenvalue $\epsilon^{2}$:\\
\\
$\phi_{14} = (a^{2}-1) {x}_{1}^{2} {x}_{2}^{2}+((1/3) a^{2}+1/3) {y}_{1}^{2}
       {x}_{2}^{2}+((4/3) a^{2}+4/3) {x}_{1} {y}_{1} {x}_{2} {y}_{2}+$\\
$((1/3)a^{2}+1/3) {x}_{1}^{2} {y}_{2}^{2}+(a^{2}-1) {y}_{1}^{2} {y}_{2}^{2}$\\

$\phi_{15} = {x}_{1}^{4}+(4 a^{2}-2) {x}_{1}^{2} {y}_{1}^{2}+{y}_{1}^{4}$\\

$\phi_{16} = {x}_{1} {x}_{2}^{3}+(2 a^{2}-1) {y}_{1} {x}_{2}^{2} {y}_{2}+(2 a^{2}-1)
       {x}_{1} {x}_{2} {y}_{2}^{2}+{y}_{1} {y}_{2}^{3}$\\

$\phi_{17}=2 {y}_{1} {x}_{2}^{5}+10 {x}_{1} {x}_{2}^{4} {y}_{2}-10 {y}_{1} {x}_{2}
       {y}_{2}^{4}-2 {x}_{1} {y}_{2}^{5}$\\

$\phi_{18} = -4 {x}_{1}^{3} {y}_{1} {x}_{2}^{2}-2 {x}_{1}^{4} {x}_{2} {y}_{2}+2
       {y}_{1}^{4} {x}_{2} {y}_{2}+4 {x}_{1} {y}_{1}^{3} {y}_{2}^{2}$
\vspace{0.25cm}
\end{tabular}

\caption{Generators of $\CC[V]^{H}$ that are eigenvectors of $\Ab(G)$ ($a$~is a 12th root of unity).}
\label{table:generatorsG4}
\end{table}

This allows us to compute a set of generators of~$\bC[V]^{[G,G]}$ consisting of eigenvectors of the $\Ab(G)$ action. The method of the proof is the same as in the case of~$S_{3}$, cf. Proposition~\ref{proposition:commutator-invariantsS3}.

\begin{proposition}
The elements $\phi_{i}$ in the Table~\ref{table:generatorsG4} are generators of the ring of invariants $\Pcal = \CC[V]^{H}$, which are eigenvectors of the action of~$\tau$.
\end{proposition}

As in the above cases, we expect that the generating set of the Cox ring of a symplectic resolution can be described as in Theorem~\ref{theorem:general-theorem}.

\begin{conjecture}
The image of the Cox ring $\Theta(\Rcal(X)) \subseteq \Pcal[t_{1}^{\pm 1}, t_{2}^{\pm 1}]$ is generated~by $$t_{1}^{-2}t_{2}, t_{1}t_{2}^{-2}, \phi_1,\ldots,\phi_8, \phi_9t_2,\ldots,\phi_{13}t_2, \phi_{14}t_1,\ldots,\phi_{18}t_1.$$
\end{conjecture}

%%%%%-----------------------------------------------------------------------

\bibliography{sympl_cox_examples} \bibliographystyle{alpha}

\begin{thebibliography}{ADHL14}

\bibitem[ADHL14]{CoxRings}
Ivan Arzhantsev, Ulrich Derenthal, J\"urgen Hausen, and Antonio Laface.
\newblock {\em Cox {R}ings}.
\newblock Cambridge University Press, New York, 2014.

\bibitem[AG10]{AG_finite}
Ivan Arzhantsev and Sergey Gaifullin.
\newblock Cox rings, semigroups and automorphisms of affine algebraic
  varieties.
\newblock {\em Sb. Math.}, 201(1):3--24, 2010.

\bibitem[AW14]{AW}
Marco Andreatta and Jarosław Wiśniewski.
\newblock 4-dimensional symplectic contractions.
\newblock {\em Geometriae Dedicata}, 168(1):311--337, 2014.

\bibitem[Bel09]{Bell}
Gwyn Bellamy.
\newblock On singular {C}alogero--{M}oser spaces.
\newblock {\em Bulletin of the London Mathematical Society}, 41(2):315--326,
  2009.

\bibitem[Ben93]{Benson}
David Benson.
\newblock {\em Polynomial invariants of finite groups}, volume 190 of {\em LMS
  Lecture Notes Series}.
\newblock Cambridge University Press, 1993.

\bibitem[BS13a]{BellamySchedler}
Gwyn Bellamy and Travis Schedler.
\newblock A new linear quotient of {${\bf C}^4$} admitting a symplectic
  resolution.
\newblock {\em Math. Z.}, 273(3-4):753--769, 2013.

\bibitem[BS13b]{BellamySchedler2}
Gwyn {Bellamy} and Travis {Schedler}.
\newblock {On the (non)existence of symplectic resolutions for imprimitive
  symplectic reflection groups}.
\newblock {\em arXiv:1309.3558 [math.SG]}, 2013.

\bibitem[CLS11]{CLS}
David~A. Cox, John~B. Little, and Hal~K. Schenck.
\newblock {\em Toric Varieties}.
\newblock Graduate studies in mathematics. American Mathematical Soc., 2011.

\bibitem[Coh80]{Cohen}
Arjeh Cohen.
\newblock Finite quaternionic reflection groups.
\newblock {\em J. Algebra}, 64(2):293 -- 324, 1980.

\bibitem[DB13]{MDB}
Maria Donten-Bury.
\newblock {Cox rings of minimal resolutions of surface quotient singularities}.
\newblock {\em arXiv:1301.2633 [math.AG], To appear in Glasgow Mathematical
  Journal}, 2013.

\bibitem[DB15]{valuationslift}
Maria Donten-Bury.
\newblock Lifting homogeneous elements with respect to monomial valuations,
  2015.
\newblock In preparation.

\bibitem[DBW14]{cox_resolutions}
Maria Donten-Bury and Jaros{\l}aw~A. Wi{\'s}niewski.
\newblock {On 81 symplectic resolutions of a 4-dimensional quotient by a group
  of order 32}.
\newblock {\em arXiv:1409.4204 [math.AG]}, 2014.

\bibitem[DGPS12]{Singular}
Wolfram Decker, Gert-Martin Greuel, Gerhard Pfister, and Hans Sch{\"o}nemann.
\newblock {\sc Singular} {3-1-5} --- {A}~computer algebra system for polynomial
  computations.
\newblock 2012.
\newblock http://www.singular.uni-kl.de.

\bibitem[FGAL11]{FGAL}
Laura Facchini, V\'{i}ctor Gonz\'{a}lez-Alonso, and Micha{\l} Laso\'{n}.
\newblock Cox rings of {D}u {V}al singularities.
\newblock {\em Matematiche (Catania)}, 66:115--136, 2011.

\bibitem[GK04]{GiKa}
Victor Ginzburg and Dmitry Kaledin.
\newblock Poisson deformations of symplectic quotient singularities.
\newblock {\em Adv. Math.}, 186(1):1--57, 2004.

\bibitem[GS13]{M2}
Daniel Grayson and Michael Stillman.
\newblock Macaulay2, a software system for research in algebraic geometry.
\newblock Available at http://www.math.uiuc.edu/Macaulay2/, 2013.

\bibitem[Ham99]{Hamm}
Helmut Hamm.
\newblock Very good quotients of toric varieties.
\newblock In {\em Real and Complex Singularities}, Chapman \& Hall/CRC Research
  Notes in Mathematics Series, pages 61--75. Taylor \& Francis, 1999.

\bibitem[IR96]{ItoReid}
Yukari Ito and Miles Reid.
\newblock The {M}c{K}ay correspondence for finite subgroups of {${\rm
  SL}(3,\bold C)$}.
\newblock In {\em Higher-dimensional complex varieties (Trento, 1994)}, pages
  221--240. de Gruyter, Berlin, 1996.

\bibitem[{Kal}02]{KaledinMcKay}
Dmitry {Kaledin}.
\newblock Mc{K}ay correspondence for symplectic quotient singularities.
\newblock {\em Invent. Math.}, 148(1):151--175, 2002.

\bibitem[LS12]{LehnSorger}
Manfred {Lehn} and Christoph {Sorger}.
\newblock {A symplectic resolution for the binary tetrahedral group}.
\newblock {\em Michel Brion (ed.), Geometric methods in representation theory.
  II. Selected papers based on the presentations at the summer school,
  Grenoble, France, June 16 – July 4, 2008.}, 24(2):429--435, 2012.

\bibitem[Ver00]{VerbitskyAsian}
Misha Verbitsky.
\newblock Holomorphic symplectic geometry and orbifold singularities.
\newblock {\em Asian J. Math.}, 4(3):553--563, 2000.

\bibitem[WW03]{WierzbaWisniewski}
Jan Wierzba and Jaros{\l}aw~A. Wi{\'s}niewski.
\newblock Small contractions of symplectic 4-folds.
\newblock {\em Duke Math. J.}, 120(1):65--95, 2003.

\end{thebibliography}

\end{document}